\documentclass[12 pt, reqno]{amsart}
\usepackage[utf8]{inputenc}
\usepackage{cite}
\usepackage{amssymb, amsfonts, amsthm,amsmath,float}
\usepackage{graphicx,a4wide, color}
\usepackage{tikz}
\usetikzlibrary{matrix,arrows,backgrounds,shapes.geometric}
\numberwithin{equation}{section}
\newtheorem{theorem}{Theorem}[section]

\newtheorem{lemma}[theorem]{Lemma}
\newtheorem{corollary}[theorem]{Corollary}
\newtheorem{proposition}[theorem]{Proposition}
\theoremstyle{definition}
\newtheorem{definition}[theorem]{Definition}
\newtheorem{remark}[theorem]{Remark}
\newtheorem{example}[theorem]{Example}

\newcounter{minutes}
\divide\time by 60
\newcounter{hours}
\multiply\time by 60 \addtocounter{minutes}{-\time}
\begin{document}

	\title[Orthogonality and duality of frames over LCA groups]
	{Orthogonality and duality of frames over\\ locally compact abelian groups}

\author{Anupam Gumber and Niraj K. Shukla} 

 \thanks{Address: Discipline of Mathematics,
	Indian Institute of Technology Indore, Simrol,   
	Indore-453 552, India. \\
   E-mail: anupamgumber.iiti@gmail.com, o.nirajshukla@gmail.com}

\begin{abstract}	 
	 	Motivated by the recent work of Bownik and Ross \cite{BR}, and Jakobsen and Lemvig \cite{JL}, this article generalizes  latest results on reproducing formulas for generalized translation invariant (GTI) systems to the setting of super-spaces over a second countable locally compact abelian (LCA) group $G$. To do so, we introduce the notion of a super-GTI system with finite sequences as generators from a super-space $L^2(G) \oplus  \cdots \oplus L^2(G) $ ($N$ summands). We characterize the generators of two super-GTI systems in the super-space such that they form a super-dual frame pair. For this, we first give necessary and sufficient conditions for two Bessel families to be orthogonal frames (we call as GTI-orthogonal frame systems) when the Bessel families have the form of GTI systems in $L^2(G)$. As a consequence, we deduce similar results for several function systems including the case of TI systems, and GTI systems on compact abelian groups. As an application, we apply our duality result for super-GTI systems to the Bessel families with a wave-packet structure (combination of wavelet as well as Gabor structure), and hence a characterization for dual super wave-packet systems on LCA groups is obtained. In addition, we relate the well established theory from literature with our results by observing several deductions in context of wavelet and Gabor systems over LCA groups with $G=\mathbb{R}^d,\mathbb{Z}^d$, etc. 
\end{abstract}


\maketitle

 
\section{Introduction}\label{sec 1} 
The concept of frames for super Hilbert spaces (or, simply super-spaces), that is, \lq\lq superframes\rq\rq, was initially introduced and investigated by Balan \cite{Bal} in the context of multiplexing.
Motivated by the wide applications of such frames in multiplexing techniques,  mobile and satellite communication, and computer area network, etc., a lot of mathematicians and engineering specialists have contributed in developing different aspects of frame properties for super-spaces (see \cite{HL,LL,LH, HKLW}). Among these properties, the orthogonality of frames in Hilbert spaces is intimately related with superframes in Hilbert spaces which  plays a key role in synthesizing superframes and frames (see \cite{W,LH,X,Xu,HL, HKLW} and references within). In this scenario, the main focus of this article is to  study orthogonal frames as well as superframes for  Hilbert spaces  associated with locally compact abelian (LCA) groups.

In the last two decades, frame theory on LCA groups has become the focus of an active research, both in theory as well as in applications due to its potential to unify the continuous theory (integral representations) and the discrete theory (series expansions). Several researchers have made remarkable contributions in establishing the theory required to analyse frame properties on such groups  (e.g., see \cite{C,JL,BR,CP, G,KL, GS, W, JL1, FHS,KLS}).

In \cite{W}, Weber studied orthogonal frames of translates in $L^2(\mathbb{R}^d)$ which lead to a characterization of superframes for $L^2(\mathbb{R}^d)$.
In this article, we plan to investigate orthogonal frames which arise from translations of generating functions via a countable family of closed, co-compact subgroups of a second countable LCA group $G$. Along with this, one of our main motive is to see applications of orthogonal frames to construct dual frames for super-spaces over LCA groups. For this, we introduce a notion of super-GTI system with generators from a super-space $L^2(G) \oplus  \cdots \oplus L^2(G)$($N$ summands). 

At this juncture, it is relevant to note that the notion of super-GTI system  generalizes the recent concept of GTI systems introduced by Jakobsen and Lemvig \cite{JL} which provides an approach that unifies the connection between the well established discrete frame theory of generalized shift invariant (GSI) systems and its continuous version. At the same time, Bownik and Ross in \cite{BR} considered the  translation invariant (TI) systems which are families with  translation along a single co-compact subgroup of an LCA group. Since GTI systems generalize TI systems, we introduce a parallel notion of super-TI systems which can be recovered from super-GTI systems.

The motivation behind the consideration of co-compact subgroups  in \cite{BR} and \cite{JL} is related to the necessity of overcoming the limitation on existence of uniform lattices for an LCA group, which says there exist LCA groups that do not contain any uniform lattices, for example, the $p$-adic numbers, whose only discrete subgroup is the neutral element which is not a uniform lattice. Another example is the $p$-adic integers which have only trivial examples of uniform lattices but have a lot of non-trivial co-compact subgroups.  Hence, the concept of co-compact subgroups in \cite{BR} and \cite{JL} respectively generalizes the work on function systems with translation along uniform lattices by Cabrelli and Paternostro \cite{CP} and Kutyniok and Labate \cite{KL}.   

In association with this, note that the work of Kutyniok and Labate \cite{KL} presented a unified theory for many of the known function systems (e.g., Gabor systems and GSI systems on 
$\mathbb{R}^d$) by introducing the notion of GSI systems in the LCA group setting. This approach is an extension of the theory of Hern$\acute{\mbox{a}}$ndez, Labate and Weiss\cite{HLW}, and  Ron and Shen \cite{RS} on GSI systems in $L^2(\mathbb{R}^d)$. 
Thus, the theory of super-GTI systems is more generalized and is applicable to a wide class of LCA groups.

Among these systems, the study of frame properties such as duality of structured function systems (e.g., Gabor, wavelet, and shearlet systems) in different settings has got special attention due to their interesting theory and enormous applications in pure mathematics as well as in engineering areas such as signal processing, image processing etc. \cite{Ba, C, LH,G,J,FKL, FHS}. 

In this scenario, we apply our characterization results on Bessel families with wave-packet, Gabor  and wavelet structure to get necessary and sufficient conditions for duals of wave-packet frames, Gabor frames and  wavelet frames in super-spaces over LCA groups. Note that one of the goals of this article is to continue the study for duals of super wavelet frames and super Gabor frames over LCA groups. For this, we first need to characterize duals of GTI systems in super-spaces  with the help of  GTI-orthogonal frame systems.
We remark that our duality results on super-GTI systems generalize the characterization of dual frames for GTI systems on LCA groups obtained in \cite[Theorem 3.4]{JL}.

Now, for discussing the main content of this article, we first recall some definitions and basic properties about continuous frames for Hilbert spaces. Such frames were introduced independently by Ali et al. \cite{AAG} and Kaiser \cite{Ka}. For a brief and self-sufficient introduction to continuous frames, we refer \cite{GH,RND}.
\begin{definition} \label{de 1.1} 
	Let $\mathcal{H}$ be a complex Hilbert space, and let ($M$, $\sum_{M}$, $\mu_M$) be a measure space, where {$\sum_{M}$ denotes the $\sigma$-algebra and $\mu_M$ the non-negative measure.
	Then,  a family of functions $\{f_m\}_{m \in M}$ in $\mathcal{H}$, is called a {\textit{continuous frame}} for $\mathcal{H}$ with respect to ($M$, $\sum_{M}$, $\mu_{M}$), if
	\begin{itemize}
		\item [(1)] $m \mapsto f_m$ is weakly measurable, that is, for all $h \in \mathcal{H}$, the mapping $M \rightarrow \mathbb{C};\, m \mapsto \langle{h, f_m}\rangle$ is measurable, and
		\item[(2)] there exist constants $0 <\alpha_1 \leq \alpha_2 $ such that
		\begin{equation}\label{eq 1.1}
		\alpha_1||h||^2 \leq \int \limits_{M}|\langle{h, f_m}\rangle|^2d\mu_{M}(m)\leq\alpha_2 ||h||^2, \,\,~\mbox{for all}~\, h \in \mathcal{H}.
		\end{equation}
	\end{itemize}
	
	The constants $\alpha_1$ and $\alpha_2$ are called continuous frame bounds.  A continuous frame $\{f_m\}_{m \in M}$ is called {\textit{tight}} if we can choose $\alpha_1 = \alpha_2$, and {\textit{Parseval}} if $\alpha_1 = \alpha_2=1$. The family $\{f_m\}_{m \in M}$} is called {\textit{Bessel}} with constant $\alpha_2$ as its {\textit{Bessel constant}} if the right side of inequality in (\ref{eq 1.1}) holds. In this case, we say that the family $\{f_m\}_{m \in M}$ satisfies the {\textit{Bessel condition}}.
\end{definition} 
Since this article  deals with only separable Hilbert spaces, we can use Petti's theorem to replace weak measurability of $m \mapsto f_m$ with (strong) measurability with respect to the Borel algebra in $\mathcal{H}$.

If $\mu_M$ is counting measure and $M=\mathbb{N}$, then $\{f_m\}_{m \in M}$ reduces to a discrete frame. In this sense continuous frames can be realized as the generalization of discrete frames. Here onwards, we will simply call continuous frames  as frames by suppressing the term continuous just for the sake of simplicity.

Given the family of functions $\mathbb{F}:=\{f_m\}_{m \in M}$, which is  Bessel  with respect to a measure space $(M, \sum_{M}, \mu_M)$, define the {\textit{synthesis  operator}} $\Theta_{\mathbb{F}}: L^2(M, \mu_M)\rightarrow \mathcal{H}$ by
\begin{align*}
\langle{\Theta_{\mathbb{F}}\varphi, h}\rangle=\int \limits_{M}  \langle{f_m,h}\rangle\varphi_m d\mu_M(m), \,\,h \in \mathcal{H},
\end{align*}
which is a well defined, linear and bounded operator \cite[Theorem 2.6]{RND}. Further, we define the adjoint of the synthesis operator as $\Theta^{\ast}_{\mathbb{F}}: \mathcal{H} \rightarrow L^2(M, \mu_M)$ given by
\begin{align*}
(\Theta^{\ast}_{\mathbb{F}} h)(m) = \langle{h, f_m}\rangle,\,\,
m \in M.
\end{align*}
 We call this operator as the {\textit{analysis operator}} of $\mathbb{F}$.

Given two Bessel families $\{f_m\}_{m \in M}$ and $\mathbb{G}:=\{g_m\}_{m \in M}$ with respect to the measure space $(M, \sum_{M}, \mu_M)$ for $\mathcal{H}$, define the {\textit{mixed dual Gramian operator}} corresponding to $\mathbb{F}$ and $\mathbb{G}$ as
\begin{align*}
\Theta_{\mathbb{G}}\Theta^{\ast}_{\mathbb{F}}: \mathcal{H} \rightarrow \mathcal{H}; \,h \mapsto \int \limits_{M} \langle{h, f_m}\rangle g_m d\mu_M(m).
\end{align*}
Gabardo and Han in \cite{GH} defined a dual frame for a continuous frame as follows:
\begin{definition}
 Let $\mathbb{F}$ and $\mathbb{G}$ be  two Bessel families with respect to the measure space $(M, \sum_{M}, \mu_M)$ for $\mathcal{H}$. We call
$\mathbb{G}$ a {\textit{dual frame}} for $\mathbb{F}$ if the following holds true:
\begin{equation} \label{eq 1.2}
	\langle{h_1, h_2}\rangle =
\int \limits_{M} \langle{h_1, f_m}\rangle \langle{g_m,h_2}\rangle d\mu_{M}(m), \,\,~\mbox{for all}~\, h_1, h_2 \in \mathcal{H}.
\end{equation}
	In this case, $\mathbb{F}$ and  $\mathbb{G}$ are actually (continuous) frames, and hence $(\mathbb{F}, \mathbb{G})$ is called a {\textit{dual frame pair}}. If $\Theta_{\mathbb{F}}$ and $\Theta_{\mathbb{G}}$ denote the synthesis operators of
	$\mathbb{F}$ and $\mathbb{G}$, respectively, then $(\ref{eq 1.2})$ is equivalent to $\Theta_{\mathbb{G}}\Theta^{\ast}_{\mathbb{F}} = I_{\mathcal{H}}$, that is, an identity operator on $\mathcal{H}$. In this case, we say that the following relation
	\begin{equation*}
	\label{eq 1.31}
h  =
	\int \limits_{M} \langle{h, f_m}\rangle g_m d\mu_{M}(m),\,~\mbox{for all}~ f \in \mathcal{H},
	\end{equation*}
holds in the weak sense.  This relation is generally known as a {\textit{reproducing formula}}  for $f \in \mathcal{H}$.	
\end{definition}	

\begin{definition}\label{de 1.4}
Suppose  $\mathbb{F}$ and $\mathbb{G}$ are Bessel families with respect to $(M, \sum_{M}, \mu_M)$ for $\mathcal{H}$. If
\begin{align*}
 \Theta_{\mathbb{G}}\Theta^{\ast}_{\mathbb{F}}:= \int \limits_{M} \langle{\cdot, f_m}\rangle g_m d\mu_M(m)=0,
\end{align*}
that is, the mixed dual Gramian operator corresponding to $\mathbb{F}$ and $\mathbb{G}$ is $0$, then the Bessel families are said to be {\textit{orthogonal}}. 
\end{definition}

The remainder of this article is organized as follows: In Section 2, we state some basic preliminaries, notation and definitions on LCA groups. We introduce the notion of super-GTI systems on LCA groups in Section 3. Along with this, we provide the statements of the main results of this article, and deduce similar results for several function systems including the case of TI systems, GSI systems and GTI systems on compact abelian groups. Section 4 forms the proof of our first main result which gives a characterization
of GTI-orthogonal frame systems in $L^2(G)$. In Section 5, we establish necessary and
sufficient conditions for the generators  of two GTI systems in the super-space over LCA groups such that they form a dual frame pair. And lastly,  we discuss  applications of our characterization results on the Bessel families with wave-packet, Gabor  and wavelet structure on LCA groups in the last section.

 
\section{Fourier analysis on locally compact abelian groups}
\label{sec 2}

In this section, we review some basic results from Fourier analysis on locally compact abelian (LCA) groups. In this way, we set up the notation used for the remainder of this article.

Here and throughout, let $G$ denote a second countable locally compact abelian (LCA) group, with the additive group composition, denoted by the symbol \lq\lq+\rq\rq\, and neutral element $0$. Note that the second countable property of $G$ is equivalent in saying that $G$ is metrizable and $\sigma$-compact. It is well known that on every LCA group $G$, there exists a {\textit{Haar measure}}, that is, a non-negative, regular Borel measure denoted as $\mu_G$ (not identically zero) which is translation invariant, $i.e.$, $\mu_G(E+x)=\mu_G(E)$ for every element $x \in G $ and every Borel set $E \subseteq G$. This measure on any LCA group is unique up to a positive constant.

Denote by $\widehat{G}$, the set of all continuous characters, that is, all continuous homomorphisms from $G$ into the torus $\mathbb{T}\cong \{z \in \mathbb{C}: |z|=1\}$. Then, under the pointwise multiplication $\widehat{G}$ forms an LCA group with unit element $1$, we call as the {\textit{dual group}} associated to $G$, when equipped with the compact convergence topology and the composition
$
(\gamma + \gamma')(x):= \gamma(x)\gamma'(x),\,\, \gamma,\,\gamma' \in \widehat{G},\,\, x \in G$,
and thus possesses a Haar measure with notation given by $\mu_{\widehat{G}}$.
It turns out that there exists a topological group isomorphism mapping the group $\widehat{\widehat{G}}$, that is, the dual group of $\widehat{G}$, onto $G$. More precisely, $\widehat{\widehat{G}}\cong G$ \cite[Pontryagin duality theorem]{F}. Note that if an LCA group $G$ is discrete then $\widehat{G}$ is compact, and vice versa.

Given an LCA group $G$ with Haar measure $\mu_G$, the integral over $G$ is translation invariant in the sense that,
\begin{align*}
\int \limits_{G} f(x+y)d\mu_{G}(x)= \int \limits_{G} f(x)d\mu_{G}(x)
\end{align*}
for each element $y\in G$ and for each Borel-measurable function $f$ on $G$. For $1\leq p< \infty$,  we define the space $L^p(G, \mu_G)$ (or, simply  $L^p(G)$) as follows:
 \begin{align*}
 L^p(G):= \Big\{f: G \rightarrow \mathbb{C} \,~\mbox{is a measurable function and}~\, \int \limits_{G} |f(x)|^p d\mu_{G}(x) < \infty  \Big\}.
 \end{align*}
 Since $G$ is a second countable LCA group, therefore, $L^p(G)$ is separable, for all $1\leq p< \infty$. In this article, we will focus only on $p=2$ case. Here, note that $L^2(G)$ is a Hilbert space with inner product given by
 \begin{align*}
 \langle{f,g}\rangle=\int \limits_{G}f(x)\overline{g(x)}d\mu_{G}(x), \,\,~\mbox{for all}~\,  f, g \in L^2(G).
 \end{align*}
 Let the Fourier transform $\,\, \widehat{}\,\,: L^{1}(G) \rightarrow C_0(\widehat{G}), f \mapsto \widehat{f}$, be defined by the operator
 \begin{align*}
\mathcal{F}f(\xi)=\widehat{f}(\xi)=\int \limits_{G}  f(x)\overline{\xi(x)} d\mu_{G}(x),\,\,\, \xi \in \widehat{G},
 \end{align*}
 where $C_0(\widehat{G})$ denotes the functions on $\widehat{G}$ vanishing at infinity. If $f\in L^1(G)$, $\widehat{f} \in \widehat{G}$, and the measures on $G$ and $\widehat{G}$ are normalized appropriately so that the Plancherel theorem holds, then the inverse Fourier transform can be defined by
 \begin{align*}
 f(x)=\mathcal{F}^{-1}\widehat{f}(x)=\int \limits_{\widehat{G}} \widehat{f}(\xi)\xi(x)d\mu_{\widehat{G}}(\xi),\,\, x \in G,
 \end{align*}
 and the Fourier transform $\mathcal{F}$ can be extended from $L^1(G) \cap L^2(G)$ to a surjective isometry  between $L^2(G)$ and $L^2(\widehat{G})$ \cite[Plancherel theorem]{F}. Thus,  the Parseval formula holds and is given by
 \begin{align*}
 \langle{f,g}\rangle=\int \limits_{G}f(x)\overline{g(x)}d\mu_{G}(x)=\int \limits_{\widehat{G}}\widehat{f}(\xi)\overline{\widehat{g}(\xi)}d\mu_{\widehat{{G}}}=\langle{\widehat{f}, \widehat{g}}\rangle, \,\,~\mbox{for all }~\,  f, g \in L^2(G).
 \end{align*}
 Let $\Gamma \subseteq G$ be a closed subgroup of an LCA group $G$. Then, the quotient $G/\Gamma$ is a regular topological group. Further, we note that it is a second countable LCA group under the quotient topology by using the fact that $G$ is second countable.
 
 For a subgroup $\Gamma$ of an LCA group $G$, the annihilator $\Gamma^{\perp}$ of $\Gamma$ is defined by
 \begin{align*}
 \Gamma^{\perp}:=\{\xi \in \widehat{G}:\xi(x)=1,\,\,\forall~\, x \in \Gamma\}.
 \end{align*}
 It follows from the definition of the topology on $\widehat{G}$ that the annihilator $\Gamma^{\perp}$ is a closed subgroup in $\widehat{G}$, and if $\Gamma$ is closed, then $(\Gamma^{\perp})^{\perp}= \Gamma$ and the following hold:
 \begin{itemize}
 	\item [(1)] There exists a topological group isomorphism mapping $\widehat{G/\Gamma}$ onto $\Gamma^{\perp}$, that is,
 	$\widehat{G/\Gamma} \cong \Gamma^{\perp}$;
 	\item [(2)] There exists a topological group isomorphism mapping $\widehat{\widehat{G}/\Gamma^{\perp}}$ onto $\Gamma$, that is, $\widehat{\widehat{G}/\Gamma^{\perp}} \cong \Gamma$.
 \end{itemize}

 The following definition will be used in this sequel:

 \begin{definition}
 	Given $G$ an LCA group, a subgroup $\Gamma$ in $G$ is said to be
 	\begin{itemize} \item[(i)] {\textit{co-compact}}  if the quotient group $G/\Gamma$ is compact.
 		\item[(ii)] a {\textit{uniform lattice}} if $\Gamma$ is discrete and quotient group $G/\Gamma$ is compact.
 		\end{itemize}
 \end{definition}
    For more information on harmonic analysis on locally compact abelian groups, we refer the reader to the classical books \cite{F, HR, HRo}.
    
\section{Notion of 
	super-generalized translation invariant systems}\label{sec 3}
We begin by considering generalized translation invariant (GTI) systems introduced by Jakobsen and Lemvig in \cite{JL}. Such systems model various discrete and continuous systems, e.g., the wavelet, shearlet and Gabor systems, etc. We refer \cite[Section 2.2]{JL}, for the following definition of the GTI system: 

\begin{definition}\label{de 2.1}
	Let $J \subset \mathbb{Z}$ be a countable index set. For each $j \in J$, let $P_j$ be a countable or an uncountable index set, let $g_{j,p} \in L^2(G)$ for $p \in P_j$, and let $\Gamma_{j}$ be a closed, co-compact  subgroup in $G$. Then, the {\textit{generalized translation invariant $($GTI$)$ system}} generated by $\{g_{j,p}\}_{p \in P_j, \,j \in J}$ with translation along closed, co-compact subgroups
	$\{\Gamma_j\}_{j \in J}$ is the family of functions given by
	$ \bigcup\limits_{j \in J}\{T_{\gamma}g_{j,p}\}_{\gamma \in \Gamma_j,\, p \in P_j},
	$ where for $y \in G$, the operator $T_{y}$, called the {\textit{translation}} by $y$, is defined by 
	\begin{align*}
	T_y: L^2(G)\rightarrow L^2(G),\, (T_yf)(x)=f(x-y),\,\, x \in G.
	\end{align*}
\end{definition}

 Now, we wish to generalize the above definition of 
GTI system to the case of super-space given by $L^2(G)\oplus \cdots \oplus L^2(G)$, that is,  the orthogonal direct sum of $L^2(G)$ with multiplicity $N$ (a natural number). Denote this space by $L^2(G)^{(N)}$, where 
\begin{align*}
L^2(G)^{(N)}:= \Big\{\bigoplus \limits_{n=1}^{N}f^{(n)} :=(f^{(1)},f^{(2)},\ldots,f^{(N)}) : f^{(n)} \in L^2(G); 1 \leq n \leq N \Big\},
\end{align*}
is a Hilbert space, we call as a {\textit{super-space}} endowed with the inner product 
\begin{align*}
\big\langle{{\bf{f}},\tilde{{\bf{f}}}}\big\rangle:= \sum\limits_{n=1}^{N} \big\langle{f^{(n)}, {\tilde{f}}^{(n)}}\big\rangle,
\end{align*}
for all ${\bf{f}}=(f^{(1)},f^{(2)},\ldots,f^{(N)}),
\tilde{{\bf{f}}}=(\tilde{f}^{(1)},\tilde{f}^{(2)},\ldots,\tilde{f}^{(N)}) \in L^2(G)^{(N)}$. In what follows, for an arbitrary ${\bf{f}} \in L^2(G)^{(N)} $, we always denote by $f^{(n)}$ its $n^{th}$ component for each $1 \leq n \leq N$.

\subsection{Definition of Super-Generalized Translation Invariant Systems} \label{sec 3.1} 

\noindent
\medskip

For each $1 \leq n \leq N$, let  $\{g^{(n)}_{j,p}\}_{p \in P_j,\,j \in J}\subset L^2(G)$, where $J$ and $P_j$  are as described in Definition~\ref{de 2.1}. Then, 
the {\textit{super-generalized translation invariant system $($super-GTI system$)$}}
generated by a collection of finite sequences  \begin{equation*}\Big\{\bigoplus \limits_{n=1}^{N}g^{(n)}_{j,p}\Big\}_{p \in P_j,\,j \in J}:=\big\{(g^{(1)}_{j,p},g^{(2)}_{j,p}, \dots, g^{(N)}_{j,p}) \big\}_{p \in P_j,\,j \in J} \subset L^2(G)^{(N)},\end{equation*} is the family of functions defined by
\begin{equation}\label{eq 2.21}
	\bigcup\limits_{j \in J}\Big\{\bigoplus \limits_{n=1}^{N} T_{\gamma}g^{(n)}_{j,p}
	\Big\}_{\gamma \in \Gamma_j,\, p \in P_j}:=	\bigcup\limits_{j \in J}\big\{T_{\gamma}g^{(1)}_{j,p}
	\oplus  \cdots \oplus T_{\gamma}g^{(N)}_{j,p}\big\}_{\gamma \in \Gamma_j,\, p \in P_j},
\end{equation}
where for each $j \in J$,  $\Gamma_j$ is a closed, co-compact subgroup in $G$.
In particular, if all $\Gamma_j$ coincide in  (\ref{eq 2.21}), that is, if $\Gamma_j=\Gamma$ (say) for each $j \in J$, then we call (\ref{eq 2.21}) as the {\textit{super-translation invariant system $($super-TI system$)$}}	in view of the fact that  ${\bf{f}}=f^{(1)}\oplus \cdots \oplus f^{(N)}$ in $\bigcup\limits_{j \in J}\Big\{\bigoplus \limits_{n=1}^{N} T_{\gamma}g^{(n)}_{j,p}
\Big\}_{\gamma \in \Gamma,\, p \in P_j}$ implies  $T_{\gamma}{\bf{f}}=T_{\gamma}f^{(1)}\oplus \cdots \oplus T_{\gamma}f^{(N)}$ is a member in  $\bigcup\limits_{j \in J}\Big\{\bigoplus \limits_{n=1}^{N} T_{\gamma}g^{(n)}_{j,p}
\Big\}_{\gamma \in \Gamma,\, p \in P_j}$ for all $\gamma \in \Gamma$.
Further, in case each $P_j$ is countable and each $\Gamma_j$ is a uniform lattice, we term the family of functions in (\ref{eq 2.21}) as the {\textit{super-generalized shift invariant system $($super-GSI system$)$}}.

\begin{remark}  We  mention that the notion of super-GTI systems, super-TI systems and super-GSI systems has not been considered before as per our knowledge, and is appearing first time in literature through this article.
It is relevant to note that the notion is more general in the sense that when $N=1$, the above mentioned systems respectively  generalize the definitions of  already existing systems such as \textit{GTI systems} \cite{JL}, \textit{TI systems} \cite{BR} and \textit{GSI systems} \cite{RS} to the case of super-spaces over LCA groups. 
\end{remark}

\subsection{Super-Generalized Translation Invariant Frame Systems}\label{sec 3.2}

\noindent
\medskip

In order to study frame properties for super-GTI systems introduced in Subsection~\ref{sec 3.1},  we need to view the family of functions (\ref{eq 2.21}) in the set-up of continuous $\mathrm{g}$-frames. Recall that these frames are a generalized version of continuous frames, more precisely,   
for a countable index
set $J \subset \mathbb{Z}$, a family of functions $\bigcup_{j \in J}\{f_{j,m}\}_{m \in M_j}$ is a {\textit{continuous generalized frame $($continuous $\mathrm{g}$-frame$)$}} for a complex Hilbert space $\mathcal{H}$ with respect to a collection of measure spaces
$\{(M_j,\sum_{M_j},\mu_j): j \in J\}$,   if
\begin{itemize}
	\item [($\mathcal{C}_1$)] $m\mapsto f_{j,m},\, M_j \rightarrow \mathcal H$ is  measurable for each $j \in J$, and
	\item[($\mathcal{C}_2$)] there exists constants $0 <\alpha_1 \leq \alpha_2 $ such that
	\begin{equation*}\label{eq 1.3}
		\alpha_1||h||^2 \leq \sum \limits_{j \in J}\int \limits_{M_j}|\langle{h, f_{j,m}}\rangle|^2d\mu_{M_j}(m)\leq\alpha_2 ||h||^2, \,\,~\mbox{for all}~\, h \in \mathcal{H}.
	\end{equation*}
\end{itemize}

Note that all the definitions and operators associated to Definition~\ref{de 1.1} can be easily visualized for the case of continuous $\mathrm{g}$-frames. For more details, we refer \cite{S, Xu} and various references within.

Our next motive is to compare the super-GTI system defined in (\ref{eq 2.21}), that is,  $\bigcup\limits_{j \in J}\Big\{\bigoplus \limits_{n=1}^{N} T_{\gamma}g^{(n)}_{j,p}
\Big\}_{\gamma \in \Gamma_j,\, p \in P_j}$  with the family of functions $\bigcup_{j \in J}\{f_{j,m}\}_{m \in M_j}$ considered in the above definition of continuous 
$\mathrm{g}$-frame. Before proceeding, we introduce some notions and notation. Let $(P_j, \sum_{P_j}, \mu_{P_j})$ and $(\Gamma_j, B_{P_j}, \mu_{
	\Gamma{_j}})$ be measure spaces for each $j \in J$, where $J \subset \mathbb{Z}$ is a countable index set and for a topological space $X$, by $B_{X}$, we denote the Borel algebra of $X$. Then, for each $j \in J$, we denote by:
\begin{itemize}
	\item [(I)]  $(\prod \limits_{N}{P}_{j})\times \Gamma_j:=({P}_j \times \cdots \times {P}_j)\times \Gamma_j $, the product measure  space formed by the Cartesian product of $\Gamma_j$ with the measure space ${P}_j \times \cdots \times {P}_j=:\prod \limits_{N}{P}_{j} $,
	\item [(II)]   $(\bigotimes \limits_{N}{\sum_{P_j}})\otimes B_{\Gamma_j}:=
	(\sum_{P_j}\otimes \cdots \otimes \sum_{P_j})\otimes B_{\Gamma_j}$, the tensor-product $\sigma$-algebra on $(\prod \limits_{N}{P}_{j})\times \Gamma_j$, formed by the tensor-product of $B_{\Gamma_j}$ with the $\sigma$-algebra
	$\sum_{P_j}\otimes \cdots \otimes \sum_{P_j} =:\bigotimes \limits_{N}\sum_{P_j}$ on 
	$\prod \limits_{N}{P}_{j}$,
	\item [(III)]  $(\mu_{\prod \limits _{N}{P}_j})\otimes \mu_{\Gamma_j}:= (\mu_{{P}_j}\otimes \cdots \otimes \mu_{{P}_j})\otimes \mu_{\Gamma_j}$, the product measure on $(\prod \limits_{N}{P}_{j})\times \Gamma_j$, formed by the tensor-product of $\mu_{\Gamma_j}$ with the  measure $\mu_{P_j}\otimes \cdots \otimes \mu_{P_j}=:\mu_{\prod \limits _{N}{P}_j}$ on $\prod \limits_{N}{P}_{j}$,
	\item [(IV)]  $\mathcal{D}_j \times \Gamma_j:=\{({\bf{p}}, \gamma): {\bf{p}}:= (p,\ldots,p) \in 
	\mathcal{D}_j, \gamma \in \Gamma_j
	\}$, a subset in $(\prod \limits_{N}{P}_{j})\times \Gamma_j$, where the notation $\mathcal{D}_j:=\{(p,\ldots,p): p \in P_j\}$ is a subset in $\prod \limits_{N}{P}_{j} $  with the subspace $\sigma$-algebra and the subspace measure defined respectively by setting $\sum_{\mathcal{D}_j}:=\{\mathcal{D}_j \cap Y: Y \in \bigotimes \limits_{N}{\sum_{P_j}}\}$ and $\mu_{\mathcal{D}_j}:=\mu^{\ast}|_{\sum_{\mathcal{D}_j}}$, that is, the function with domain $\sum_{\mathcal{D}_j}$ such that $\mu_{\mathcal{D}_j}(B_j)=\mu^{\ast}(B_j)$ for every $B_j \in \sum_{\mathcal{D}_j}$, where  $\mu^{\ast}$  defined for any $\mathcal{S}\subseteq \prod \limits_{N}{P}_{j} $ by $
	\mu^{\ast}(\mathcal{S}):=\inf\big\{\mu_{\prod \limits _{N}{P}_j}(E): E \in \bigotimes \limits_{N}{\sum_{P_j}}, \mathcal{S} \subseteq E\big\}
	$
	denotes an outer measure on  $\prod \limits_{N}{P}_{j}$, provided the infimum (inf) exists. In this case, we call $(\mathcal{D}_j, \sum_{\mathcal{D}_j}, \mu_{\mathcal{D}_j})$ as a subspace of the product measure space $\big(\prod \limits_{N}{P}_{j},\, \bigotimes \limits_{N}\sum_{P_j},\, \mu_{\prod \limits _{N}{P}_j}\big)$.	
	\item [(V)]  $\sum_{\mathcal{D}_j}\otimes B_{\Gamma_j}$ and $\mu_{\mathcal{D}_j}\otimes \mu_{\Gamma_j}$,  the notation for the subspace $\sigma$-algebra and the subspace measure on $\mathcal{D}_j \times \Gamma_j$, where the $\sigma$-algebra and the measure can be defined by using the technique described in (IV).
	In this case, we say $(\mathcal{D}_j \times \Gamma_j,\sum_{\mathcal{D}_j}\otimes B_{\Gamma_j}, \mu_{\mathcal{D}_j}\otimes \mu_{\Gamma_j})$ is a subspace of the product measure space $\Big(\big(\prod \limits_{N}{P}_{j}\big)\times \Gamma_j,\, \big(\bigotimes \limits_{N}\sum_{P_j}\big)\otimes B_{\Gamma_j},\, \big(\mu_{\prod \limits _{N}{P}_j}\big) \otimes \mu_{\Gamma_j}\Big).$
\end{itemize}

From the above discussion, it follows that we can view the super-GTI system defined in (\ref{eq 2.21})(when compared to the set-up of a continuous $\mathrm{g}$-frame) as a family of functions in  $L^2(G)^{(N)}$ with respect to the collection of measure spaces 
$\{(M_j,\sum_{M_j},\mu_j):=\big(\mathcal{D}_j \times \Gamma_j,\sum_{\mathcal{D}_j}\otimes B_{\Gamma_j}, \mu_{\mathcal{D}_j}\otimes \mu_{\Gamma_j}\big): j \in J\}$.

\medskip
\noindent
{{\bf{Standing Hypotheses}}}: To investigate frame properties for super-GTI systems considered in (\ref{eq 2.21}), we assume that these systems satisfy the following criterion for the rest of this article. For each $j \in J$:
\begin{itemize}
	\item [(1)]
	$($$\mathcal{D}_{j},\, {\sum _{\mathcal{D}_j}},\, \mu_{\mathcal{D}_j} $$)$ is a $\sigma$-finite measure space,
	\item [(2)]  the mapping ${\bf{p}} \mapsto \bigoplus \limits_{n=1}^{N}g^{(n)}_{j,p}$, $($$\mathcal{D}_{j},\,{\sum_{\mathcal{D}_j}} $$)$ $\rightarrow$ $($$L^2(G)^{(N)},\, B_{L^2(G)^{(N)}}$$)$ is measurable,
	\item [(3)] the mapping 
	$($${\bf{p}},x$$)$ $\mapsto \bigoplus \limits_{n=1}^{N}g^{(n)}_{j,p}$($x$), that is, 
	$(\mathcal{D}_{j} \times G,\,{\sum _{\mathcal{D}_j}} \otimes B_G) \rightarrow (\mathbb{C}^N, B_{\mathbb{C}^{N}})
	$
	is measurable.
\end{itemize}
\begin{remark} Clearly, if we consider $N=1$, then the measure spaces $\big(\mathcal{D}_j \times \Gamma_j,\sum_{\mathcal{D}_j}\otimes B_{\Gamma_j}, \mu_{\mathcal{D}_j}\otimes \mu_{\Gamma_j}\big)$ and $(\mathcal{D}_j, \sum_{\mathcal{D}_j}, \mu_{\mathcal{D}_j})$ respectively coincide with $\big({P}_{j}\times \Gamma_j,\, \sum_{P_j}\otimes B_{\Gamma_j},\, \mu_{{P}_j} \otimes \mu_{\Gamma_j}\big)$ and $({P}_{j}, \sum_{P_j}, \mu_{{P}_j})$  for each $j$, and hence the super-GTI system introduced in (\ref{eq 2.21}) represents the generalized form of the system considered in the Definition~\ref{de 2.1}.
\end{remark}

\noindent
{\textit{Super-GTI System as a Continuous $\mathrm{g}$}}-frame: For this, we first verify the condition ($\mathcal{C}_1$). Let $j \in J$. Consider a function $F: \mathcal{D}_j\times \Gamma_j \rightarrow L^2(G)^{(N)};\,\,  ({\bf{p}},\gamma) \mapsto \bigoplus \limits_{n=1}^{N} T_{\gamma}g^{(n)}_{j,p}$. The function $F$ is continuous in $\gamma$ and measurable in ${\bf{p}}$, and hence represents a Carath\'eodory function 
$\widetilde{F}$ which is
defined on
$\mathcal{D}_j$
by
$\widetilde{F}({\bf{p}})(\gamma)=F({\bf{p}}, \gamma)$. Since $\Gamma_j \subset G$ is second countable and  locally compact, and $L^2(G)^{(N)}$ is separable, it follows that $\widetilde{F}$, and hence the function $F$  is jointly measurable on $(M_j, \sum_{M_j})=\big(\mathcal{D}_j \times \Gamma_j,\sum_{\mathcal{D}_j}\otimes B_{\Gamma_j}\big)$. Thus, the condition ($\mathcal{C}_1$) holds, and the super-GTI system (\ref{eq 2.21}) is automatically weakly measurable. In addition, if the super-GTI system  satisfies the condition ($\mathcal{C}_2$) with respect to the measure spaces $\big(\mathcal{D}_j \times \Gamma_j,\sum_{\mathcal{D}_j}\otimes B_{\Gamma_j}, \mu_{\mathcal{D}_j}\otimes \mu_{\Gamma_j}\big)$, we call (\ref{eq 2.21}) as the \textit{super-generalized translation frame system $($super-GTI frame system$)$} for $L^2(G)^N$. Similar conclusions can be drawn for the case of super-GTI systems being Bessel families, Parseval frames, etc.

\medskip
\noindent
{{\bf{Local Integrability Conditions:}}}
We mention that for stating our main characterization results in Subsection~\ref{sec 3.2}, we require some technical definition in the form of a local integrability condition. For the case of GSI systems, such condition was originally introduced by   Hern$\acute{\mbox{a}}$ndez, Labate and Weiss in \cite{HLW} for $L^2(\mathbb{R}^n)$, and later generalized by  Kutyniok and Labate in \cite{KL} for $L^2(G)$. This condition was further proposed in a more generalized form by Jakobsen and Lemvig in \cite{JL} for GTI systems in $L^2(G)$. We state these conditions as follows:
	 \begin{definition}\label{de 2.2}
	 	Consider two GTI systems $\bigcup\limits_{j \in J}\{T_{\gamma}g_{j,p}\}_{\gamma \in \Gamma_j, p \in P_j}$ and 
	 	$\bigcup\limits_{j \in J}\{T_{\gamma}h_{j,p}\}_{\gamma \in \Gamma_j, p \in P_j}$ in $L^2(G)$. 
	 	\begin{itemize}
	 		\item [(i)] We say that 
	 		$\bigcup\limits_{j \in J}\{T_{\gamma}g_{j,p}\}_{\gamma \in \Gamma_j, p \in P_j}$ satisfies the \textit{local integrability condition $($LIC$)$} if \begin{equation}\label{eq 3.211}
	 		\sum \limits_{j \in J}\int\limits_{P_j}\sum\limits_{\alpha \in \Gamma^{\perp}_{j}}\int\limits_{supp\,\widehat{f}}|\widehat{f}(\xi+\alpha)\widehat{g}_{j,p}(\xi)|^2d\mu_{\widehat{G}}(\xi)d\mu_{P_j}(p)<\infty,\,~\mbox{for all}~ f \in \mathfrak{D},
	 		\end{equation}
	 		where for a Borel set $B$ in $\widehat{G}$ with $\mu_{\widehat{G}}(\overline{B})=0$, we define  the subset $\mathfrak{D}$ in $L^2(G)$ as follows:
	 		\begin{align*}
	 		\mathfrak{D}:= \{f \in L^2(G): \widehat{f}\in L^{\infty}(\widehat{G})\, ~\mbox{and supp} \widehat{f} ~\mbox{is compact in}~ \widehat{G}\setminus B \}.
	 		\end{align*}
	 		\item [(ii)] $\bigcup\limits_{j \in J}\{T_{\gamma}g_{j,p}\}_{\gamma \in \Gamma_j, p \in P_j}$ and 
	 		$\bigcup\limits_{j \in J}\{T_{\gamma}h_{j,p}\}_{\gamma \in \Gamma_j, p \in P_j}$ satisfy
	 		the {\textit{dual $\alpha$-local integrability condition $($dual $\alpha$-LIC$)$}} if
	 		\begin{equation}\label{eq 3.1111}
	 		\sum \limits_{j \in J}\int\limits_{P_j}\sum\limits_{\alpha \in \Gamma^{\perp}_{j}}\int\limits_{\widehat{G}}|\widehat{f}(\xi)\widehat{f}(\xi+\alpha)\widehat{g}_{j,p}(\xi)\widehat{h}_{j,p}(\xi +\alpha)|d\mu_{\widehat{G}}(\xi)d\mu_{P_j}(p)<\infty,\,~\mbox{for all}~f \in \mathfrak{D}.
	 		\end{equation}
	 		In case $g_{j,p}=h_{j,p}$ for each $j$ and $p$, we refer to (\ref{eq 3.1111}) as the {\textit{$\alpha$-local integrability condition $($$\alpha$-LIC$)$}} for the GTI system $\bigcup\limits_{j \in J}\{T_{\gamma}g_{j,p}\}_{\gamma \in \Gamma_j, p \in P_j}$.
	 		Note that the integrands in (\ref{eq 3.211}) and (\ref{eq 3.1111}) are measurable on $P_j \times \widehat{G}$, therefore, we are allowed to reorder sums and integrals in the local integrability conditions. 	
	 	\end{itemize}

	 	\begin{remark} \label{re 3.3}
	 		In view of \cite [Lemma 3.9]{JL}, it is clear that 
	 		\begin{itemize}
	 			\item[(i)] LIC implies the $\alpha$-LIC while the converse need not be true (e.g., see \cite[Example 1]{JL}). 
	 			\item[(ii)] If two GTI systems satisfy the LIC, then they satisfy the dual $\alpha$-LIC.
	 		\end{itemize}
	 	\end{remark}
	 	
	 \end{definition}
 Note that the subset $\mathfrak{D}$ defined above is dense in $L^2(G)$, and since it is sufficient to prove the various frame properties on the dense subset of a Hilbert space, we may verify our results for $\mathfrak{D}$ and then extend on $L^2(G)$ by a density argument.	 
\subsection{Main Results} \label{sec 3.3} 

\noindent
\medskip

In this subsection, we state our main results discussed in the article. The first one is Theorem~\ref{thm 2.1} which provides a characterization of GTI-orthogonal frame systems on LCA groups (proof shall be discussed in Section 4).  
It is known that orthogonal frames for a Hilbert space play a key role in characterizing Parseval frame and dual frames for super-spaces \cite{HL,LH,W,X,Xu}. In our setting, we define such frames as GTI systems satisfying a special case of Definition~\ref{de 1.4} as follows:
	\begin{definition}\label{de 2.4}
		Let  $\bigcup\limits_{j \in J}\{T_{\gamma}g_{j,p}\}_{\gamma \in \Gamma_j,\, p \in P_j}$ and $\bigcup\limits_{j \in J}\{T_{\gamma}h_{j,p}\}_{\gamma \in \Gamma_j,\, p \in P_j}$ be Bessel families (frames for $L^2(G)$). Then, we term these systems as {\textit{GTI-orthogonal Bessel systems $($GTI-orthogonal frame systems$)$}} in $L^2(G)$ if they are orthogonal. In particular, by replacing GTI systems with TI systems and GSI systems, this definition corresponds to  {\textit{TI-orthogonal Bessel systems $($TI-orthogonal frame systems$)$}} and {\textit{GSI-orthogonal Bessel systems $($GSI-orthogonal frame systems$)$}}, respectively.
	\end{definition}
	
Next, we provide the statement of our first main result. To the best of our knowledge, we realized that the characterization results for orthogonal frames have not been studied earlier  in the context of LCA groups. Moreover, for the set-up of LCA groups, orthogonal frames in terms of GTI systems, TI systems and GSI systems are appearing first time in the literature via this article.
\begin{theorem} \label{thm 2.1}
Let  $\bigcup\limits_{j \in J}\{T_{\gamma}g_{j,p}\}_{\gamma \in \Gamma_j,\, p \in P_j}$ and $\bigcup\limits_{j \in J}\{T_{\gamma}h_{j,p}\}_{\gamma \in \Gamma_j,\, p \in P_j}$ be Bessel families $($frames for $L^2(G)$$)$ which satisfy the dual $\alpha$-LIC. Then, the following assertions are equivalent: 
	\begin{itemize}
		\item[(i)] $\bigcup\limits_{j \in J}\{T_{\gamma}g_{j,p}\}_{\gamma \in \Gamma_j,\, p \in P_j}$ and $\bigcup\limits_{j \in J}\{T_{\gamma}h_{j,p}\}_{\gamma \in \Gamma_j,\, p \in P_j}$ are GTI-orthogonal Bessel systems $($GTI-orthogonal frame systems$)$ in $L^2(G)$,
		\item[(ii)] for each $\alpha \in \bigcup\limits_{j \in J}\Gamma_j^{\perp}\setminus \{0\}$, we have 
		\begin{equation}\label{eq 2.1}
		\sum\limits_{j \in J: \alpha \in \Gamma_j^{\perp} } \int \limits_{P_j} \overline{\widehat {h}_{j,p}(\xi)} \widehat{g}_{j,p}(\xi+  \alpha)d{\mu}_{P_j}(p)=0,\, ~\mbox{for a.e.}~\, \xi \in \widehat{G},
		\end{equation}
		and 
		\begin{equation}\label{eq 2.2}
		\sum\limits_{j \in J} \int \limits_{P_j} \overline{\widehat {h}_{j,p}(\xi)} \widehat{g}_{j,p}(\xi)d{\mu}_{P_j}(p)=0,\, ~\mbox{for a.e.}~\, \xi \in \widehat{G}.
		\end{equation}
	\end{itemize}
	
\end{theorem}
 Observe that the above statement can be used to deduce the following characterization results  corresponding to TI-orthogonal frame systems, GSI-orthogonal frame systems and  GTI-orthogonal frame systems (over a compact abelian group):

\medskip
\noindent
{\textit{For TI Systems}}: Let the closed and co-compact subgroups $\Gamma_j$ in the definition of the GTI system be the same for each $j \in J$, that means, $\Gamma_j=\Gamma$ (say). Then, the GTI system reduces to the \textit{translation invariant system $($TI system$)$} investigated by Bownik and Ross in \cite{BR}. In this case,  Theorem~\ref{thm 2.1} leads to the following result which can be easily deduced by observing that the LIC condition is automatically satisfied in view of the Bessel condition on two TI systems $\bigcup\limits_{j \in J}\{T_{\gamma}g_{j,p}\}_{\gamma \in \Gamma,\, p \in P_j}$ and $\bigcup\limits_{j \in J}\{T_{\gamma}h_{j,p}\}_{\gamma \in \Gamma,\, p \in P_j}$ along with the technique followed in the proof of \cite[Theorem 3.11]{JL}, and hence the dual $\alpha$-LIC holds on the TI systems by using Remark~\ref{re 3.3}. Thus, we obtain the following result: 
\begin{corollary} \label{coro 3.9}
	For a co-compact subgroup $\Gamma$ in $G$, let the two TI systems $\bigcup\limits_{j \in J}\{T_{\gamma}g_{j,p}\}_{\gamma \in \Gamma,\, p \in P_j}$ and $\bigcup\limits_{j \in J}\{T_{\gamma}h_{j,p}\}_{\gamma \in \Gamma,\, p \in P_j}$ be Bessel families $($frames for $L^2(G)$$)$. Then, the following  are equivalent: 
	\begin{itemize}
		\item[(i)] $\bigcup\limits_{j \in J}\{T_{\gamma}g_{j,p}\}_{\gamma \in \Gamma,\, p \in P_j}$ and $\bigcup\limits_{j \in J}\{T_{\gamma}h_{j,p}\}_{\gamma \in \Gamma,\, p \in P_j}$ are TI-orthogonal Bessel systems $($TI-orthogonal frame systems$)$ in $L^2(G)$,
		\item[(ii)] for each $\alpha \in \Gamma^{\perp}$, we have 
		\begin{equation}\label{eq 2.1}
		\sum\limits_{j \in J} \int \limits_{P_j} \overline{\widehat {h}_{j,p}(\xi)} \widehat{g}_{j,p}(\xi+  \alpha)d{\mu}_{P_j}(p)=0,\, ~\mbox{for a.e.}~\, \xi \in \widehat{G}.
		\end{equation}
	
	\end{itemize}
	
\end{corollary}
\noindent
{\textit{For GSI Systems}}: For each $j \in J$, by using $\Gamma_j$ as a uniform lattice (that is, a discrete, co-compact subgroup) in the GTI system, we arrive at the {\textit{generalized shift invariant $($GSI system$)$}} (see, \cite{HLW, KL}). Then,  there exists a compact fundamental domain $\mathbb{U}_j\subset G$ corresponding to $\Gamma_j$ for each $j$, with $G=\mathbb{U}_j\Gamma_j$ which says that for any $x \in G$ we can write $x=u\gamma$, where $u \in \mathbb{U}_j, \gamma \in \Gamma_j$ are unique. In this regard, we have the following deduction from Theorem~\ref{thm 2.1}:  
\begin{corollary} \label{coro 3.10}
	For each $j \in J$, let $\Gamma_j$ be a uniform lattice in $G$, and let the two GSI systems $\bigcup\limits_{j  \in J}\{T_{\gamma}g_{j,p}\}_{\gamma \in \Gamma_j,\, p \in P_j}$ and $\bigcup\limits_{j \in J}\{T_{\gamma}h_{j,p}\}_{\gamma \in \Gamma_j,\, p \in P_j}$ be Bessel families $($frames for $L^2(G)$$)$ satisfying the dual $\alpha$-LIC. Then, the following statements are equivalent: 
	\begin{itemize}
		\item[(i)] $\bigcup\limits_{j \in J}\{T_{\gamma}g_{j,p}\}_{\gamma \in \Gamma_j,\, p \in P_j}$ and $\bigcup\limits_{j \in J}\{T_{\gamma}h_{j,p}\}_{\gamma \in \Gamma_j,\, p \in P_j}$ are GSI-orthogonal Bessel systems $($GSI-orthogonal frame systems$)$ in $L^2(G)$, that is,
		\begin{align*}
		\sum \limits_{j \in J: \alpha \in {\Gamma_j}^{\perp}} \int \limits_{P_j} V(\Gamma_j) \sum\limits_{\gamma \in \Gamma_j}\langle{f, T_{\gamma}g_{j,p}}\rangle T_{\gamma}h_{j,p}d\mu_{P_j}(p)=0,\,\,~\mbox{for all}~ \, f \in L^2(G),
		\end{align*}
		where for each $j$, the symbol $V(\Gamma_j):= \mu_{G}(\mathbb{U}_j)$ denotes the   lattice size with $\mathbb{U}_j \subset G$ as a compact fundamental domain for $\Gamma_j$. 
		\item[(ii)] for each $\alpha \in \bigcup\limits_{j \in J}\Gamma_j^{\perp}$, we have 
		\begin{equation}\label{eq 2.11}
		\sum \limits_{j \in J: \alpha \in {\Gamma_j}^{\perp}} \int \limits_{P_j} \overline{\widehat {h}_{j,p}(\xi)} \widehat{g}_{j,p}(\xi+  \alpha)d{\mu}_{P_j}(p)=0,\, ~\mbox{for a.e.}~\,\,\, \xi \in \widehat{G}.
		\end{equation}
	
	\end{itemize}
	
\end{corollary}

\noindent
\textit{For Compact Abelian Groups}:
Note that GTI systems over compact abelian groups which are Bessel families, satisfy the dual $\alpha$-LIC,  in view of the fact on LIC proved in the proof of \cite[Theorem 3.14]{JL} along with the Remark~\ref{re 3.3}. In this direction, we obtain the following characterization result from Theorem~\ref{thm 2.1}:
\begin{corollary} \label{coro 3.11}
	Let $G$ be a compact abelian group, and let the two GTI systems $\bigcup\limits_{j \in J}\{T_{\gamma}g_{j,p}\}_{\gamma \in \Gamma_j,\, p \in P_j}$ and $\bigcup\limits_{j \in J}\{T_{\gamma}h_{j,p}\}_{\gamma \in \Gamma_j,\, p \in P_j}$ be Bessel families $($frames for $L^2(G)$$)$. Then, the following assertions are equivalent: 
	\begin{itemize}
		\item[(i)] $\bigcup\limits_{j \in J}\{T_{\gamma}g_{j,p}\}_{\gamma \in \Gamma_j,\, p \in P_j}$ and $\bigcup\limits_{j \in J}\{T_{\gamma}h_{j,p}\}_{\gamma \in \Gamma_j,\, p \in P_j}$ are GTI-orthogonal Bessel systems $($GTI-orthogonal frame systems$)$ in $L^2(G)$,
		\item[(ii)] for each $\alpha \in \bigcup\limits_{j \in J}\Gamma_j^{\perp}$, we have 
		\begin{equation}\label{eq 2.111}
		\sum\limits_{j \in J: \alpha \in \Gamma_j^{\perp} } \int \limits_{P_j} \overline{\widehat {h}_{j,p}(\xi)} \widehat{g}_{j,p}(\xi+  \alpha)d{\mu}_{P_j}(p)=0,\, ~\mbox{for a.e.}~\,\, \xi \in \widehat{G}.
		\end{equation}

	\end{itemize}
\end{corollary}

Next, we wish to state our second main result, that is,   Theorem~\ref{thm 4.51} which characterizes the generators of two super-GTI systems in $L^2(G)^{(N)}$ such that they form a dual frame pair (proof shall be discussed in Section 5). Here, we would like to add that Theorem~\ref{thm 4.51}  generalizes the recent  result for duals of GTI systems \cite[Theorem 3.4]{JL} to the set-up of super-space over an LCA group, and hence Theorem~\ref{thm 4.51} can be used to deduce the characterization results for duals of super-TI systems, super-GSI systems etc. $($which are also new to the literature in the context of LCA groups$)$. In particular, our result  generalizes the existing results on duals for special structured systems such as wave-packet systems, Gabor and wavelet systems $($see Section 6 for more details$)$ . Moreover, we can easily deduce the duality conditions in case of $G=\mathbb{R}^d, \mathbb{Z}^d$, etc. Note that the above mentioned results hold equally for Parseval frames, but in that case the Bessel family assumption is not needed. Before proceeding further, we give the following definition in this sequel: 
\begin{definition}\label{de 2.6}
	For each $1 \leq n \leq N$, let    $\{g^{(n)}_{j,p}\}_{p \in P_j,\,j \in J}$ and $\{h^{(n)}_{j,p}\}_{p \in P_j,\,j \in J}$ be subsets in $L^2(G)$. Then, we say that $\big\{\bigoplus \limits_{n=1}^{N} g^{(n)}_{j,p}\big\}_{p \in P_j,\,j \in J}$ and $\big\{\bigoplus \limits_{n=1}^{N} h^{(n)}_{j,p}\big\}_{p \in P_j,\,j \in J}$ form a {\textit{super-dual frame pair}} in $L^2(G)^{(N)}$ if the super-GTI systems  $\bigcup\limits_{j \in J}\big\{\bigoplus \limits_{n=1}^{N} T_{\gamma}g^{(n)}_{j,p}
	\big\}_{\gamma \in \Gamma_j,\, p \in P_j}$ and  $\bigcup\limits_{j \in J}\big\{\bigoplus \limits_{n=1}^{N} T_{\gamma}h^{(n)}_{j,p}
	\big\}_{\gamma \in \Gamma_j,\, p \in P_j}$ satisfying the Bessel condition, are dual frames for the super-space $L^2(G)^{(N)}$. In this case, we term the super-GTI system  $\bigcup\limits_{j \in J}\big\{\bigoplus \limits_{n=1}^{N} T_{\gamma}g^{(n)}_{j,p}
	\big\}_{\gamma \in \Gamma_j,\, p \in P_j}$ as a {\textit{dual super-GTI frame system}} for  $\bigcup\limits_{j \in J}\big\{\bigoplus \limits_{n=1}^{N} T_{\gamma}h^{(n)}_{j,p}
	\big\}_{\gamma \in \Gamma_j,\, p \in P_j}$, and vice versa.
\end{definition}
 The following is our second main result which provides the conditions on two super-GTI systems to form dual frames for $L^2(G)^{(N)}$: 
\begin{theorem}\label{thm 4.51}
	For each $1 \leq n \leq N$, let  $\bigcup\limits_{j \in J}\{ T_{\gamma}g^{(n)}_{j,p}
	\}_{\gamma \in \Gamma_j,\, p \in P_j}$ and $\bigcup\limits_{j \in J} \{T_{\gamma}h^{(n)}_{j,p}
	\}_{\gamma \in \Gamma_j,\, p \in P_j}$ be Bessel families satisfying the dual $\alpha$-LIC. Then, $\big\{\bigoplus \limits_{n=1}^{N} g^{(n)}_{j,p}\big\}_{p \in P_j,\,j \in J}$ and $\big\{\bigoplus \limits_{n=1}^{N} h^{(n)}_{j,p}\big\}_{p \in P_j,\,j \in J}$  form a super-dual frame pair in $L^2(G)^{(N)}$ if, and only if, both of the following  hold:
 
	\begin{itemize}	
		\item[(i)] 
		for each $1 \leq n \leq N$ and  $\alpha \in \bigcup\limits_{j \in J}\Gamma_j^{\perp}$, we have
		\begin{equation}\label{eq 2.51}  
\sum\limits_{j \in J: \alpha \in \Gamma_j^{\perp} } \int \limits_{P_j} \overline{\widehat {h}_{j,p}^{(n)}(\xi)} \widehat{g}_{j,p}^{(n)}(\xi+  \alpha)d{\mu}_{P_j}(p)=\delta_{\alpha,0},\, ~\mbox{for a.e.}~\,\, \xi \in \widehat{G},
		\end{equation}
		\item[(ii)] for each $1 \leq n_1\neq n_2 \leq N$ and  $\alpha \in \bigcup\limits_{j \in J}\Gamma_j^{\perp}$, we have 
		\begin{equation}\label{eq 2.61}
		\sum\limits_{j \in J: \alpha \in \Gamma_j^{\perp} } \int \limits_{P_j} \overline{\widehat {h}_{j,p}^{(n_1)}(\xi)} \widehat{g}_{j,p}^{(n_2)}(\xi+  \alpha)d{\mu}_{P_j}(p)=0,\, ~\mbox{for a.e.}~\,\, \xi \in \widehat{G}.
		\end{equation}
	\end{itemize}

\end{theorem}

Note that  Theorem~\ref{thm 4.51} can be used to deduce the duality results for super-TI systems, super-GSI systems, and  super-GTI systems (with $G$ as a compact abelian group) by following the same technique which we have used to verify local integrability conditions in Corollory~\ref{coro 3.9},  Corollory~\ref{coro 3.10},
and Corollory~\ref{coro 3.11}, respectively. Along with this, we can easily obtain the corresponding characterization for Parseval frames in super-spaces over LCA groups by using Theorem~\ref{thm 4.51} and by removing the Bessel family assumption on the GTI system:

\begin{corollary}
	For each $1 \leq n \leq N$, let the GTI system  $\bigcup\limits_{j \in J}\{ T_{\gamma}g^{(n)}_{j,p}
	\}_{\gamma \in \Gamma_j,\, p \in P_j}$ 
	satisfy the  $\alpha$-LIC. Then, the super-GTI system generated by $\big\{\bigoplus \limits_{n=1}^{N} g^{(n)}_{j,p}\big\}_{p \in P_j,\,j \in J}$  forms a Parseval frame for $L^2(G)^{(N)}$ if, and only if, both of the following  hold:

	\begin{itemize}	
		\item[(i)] 
		for each $1 \leq n \leq N$ and  $\alpha \in \bigcup\limits_{j \in J}\Gamma_j^{\perp}$, we have
		\begin{equation*}\label{eq 3.13}  
		\sum\limits_{j \in J: \alpha \in \Gamma_j^{\perp} } \int \limits_{P_j} \overline{\widehat {g}_{j,p}^{(n)}(\xi)} \widehat{g}_{j,p}^{(n)}(\xi+  \alpha)d{\mu}_{P_j}(p)=\delta_{\alpha,0},\, ~\mbox{for a.e.}~\, \xi \in \widehat{G},
		\end{equation*}
		\item[(ii)] for each $1 \leq n_1\neq n_2 \leq N$ and  $\alpha \in \bigcup\limits_{j \in J}\Gamma_j^{\perp}$, we have 
		\begin{equation*}\label{eq 3.14}
		\sum\limits_{j \in J: \alpha \in \Gamma_j^{\perp} } \int \limits_{P_j} \overline{\widehat {g}_{j,p}^{(n_1)}(\xi)} \widehat{g}_{j,p}^{(n_2)}(\xi+  \alpha)d{\mu}_{P_j}(p)=0,\, ~\mbox{for a.e.}~\, \xi \in \widehat{G}.
		\end{equation*}
	
	\end{itemize}
\end{corollary}	
  
 \section{ A characterization result for GTI-orthogonal frame systems}\label{sec 4}
 
 In the present section, we obtain a proof 
 for Theorem~ \ref{thm 2.1}, that gives  necessary and sufficient conditions for two GTI systems to form orthogonal frames for $L^2(G)$. For this, the following result plays an important role:

 \begin{theorem}\label{thm 3.1}
 	Suppose  $\bigcup\limits_{j \in J}\{T_{\gamma}g_{j,p}\}_{\gamma \in \Gamma_j, p \in P_j}$ and $\bigcup\limits_{j \in J}\{T_{\gamma}h_{j,p}\}_{\gamma \in \Gamma_j, p \in P_j}$ are Bessel families satisfying the dual $\alpha$-LIC. Then, the following statements are equivalent:
 	\begin{itemize}
 		\item[(i)] 	the mixed dual Gramian operator  corresponding to  $\bigcup\limits_{j \in J}\{T_{\gamma}g_{j,p}\}_{\gamma \in \Gamma_j, p \in P_j}$, and $\bigcup\limits_{j \in J}\{T_{\gamma}h_{j,p}\}_{\gamma \in \Gamma_j, p \in P_j}$  commutes with the family of translations $\{T_{x}\}_{x \in G}$,
 		\item [(ii)] for each $\alpha \in \bigcup\limits_{j \in J}\Gamma_j^{\perp}\setminus \{0\}$, 
 			\begin{equation}\label{eq 3.1}
 		t_{\alpha}(\xi):=	\sum\limits_{j \in J: \alpha \in \Gamma_j^{\perp} } \int \limits_{P_j} \overline{\widehat {h}_{j,p}(\xi)} \widehat{g}_{j,p}(\xi+  \alpha)d{\mu}_{P_j}(p)=0,\, ~\mbox{for a.e.}~\, \xi \in \widehat{G}.
 			\end{equation}
 				
 	\end{itemize}
Moreover,  if $($i$)$ or $($ii$)$ holds, then the mixed dual Gramian operator is a Fourier multiplier whose symbol is 
 \begin{equation}\label{eq 3.2}
 s(\xi) =\sum\limits_{j \in J} \int \limits_{P_j} \overline{\widehat {h}_{j,p}(\xi)} \widehat{g}_{j,p}(\xi)d{\mu}_{P_j}(p),\, ~\mbox{for a.e.}~\, \xi \in \widehat{G}.
 \end{equation}
 \end{theorem}

We first remark that the equations (\ref{eq 3.1})
and (\ref{eq 3.2}) are well
defined which can be easily verified by using Cauchy-Schwarz ineqality in the following computation:
 \begin{align*}
 \sum\limits_{j \in J: \alpha \in \Gamma_j^{\perp} } \int \limits_{P_j} |\overline{\widehat {h}_{j,p}(\xi)} \widehat{g}_{j,p}(\xi+  \alpha)|d{\mu}_{P_j}(p) &\leq 
 \sum\limits_{j \in J } \int \limits_{P_j} |{\widehat {h}_{j,p}(\xi)}| |\widehat{g}_{j,p}(\xi+  \alpha)|d{\mu}_{P_j}(p)\\
 &\leq 
 \sum\limits_{j \in J } \Big(\int \limits_{P_j} |{\widehat {h}_{j,p}(\xi)}|^2d{\mu}_{P_j}(p)\Big)^{1/2} \Big(\int \limits_{P_j}|\widehat{g}_{j,p}(\xi+  \alpha)|^2d{\mu}_{P_j}(p)\Big)^{1/2}\\
 &\leq 
 \Big(\sum\limits_{j \in J } \int \limits_{P_j} |{\widehat {h}_{j,p}(\xi)}|^2d{\mu}_{P_j}(p)\Big)^{1/2} \Big(\sum\limits_{j \in J }\int \limits_{P_j}|\widehat{g}_{j,p}(\xi+  \alpha)|^2d{\mu}_{P_j}(p)\Big)^{1/2},
 \end{align*}
 and hence,  we can write
 \begin{equation}\label{eq 4.3}
 \sum\limits_{j \in J: \alpha \in \Gamma_j^{\perp} } \int \limits_{P_j} |\overline{\widehat {h}_{j,p}(\xi)} \widehat{g}_{j,p}(\xi+  \alpha)|d{\mu}_{P_j}(p)
 \leq \beta ,\,\, ~\mbox{for a.e.}~\, \xi \in \widehat{G},
 \end{equation}
in view of  \cite[Proposition 3.3]{JL} and by letting $\beta$ as a common Bessel constant for the two GTI systems.
Now, in order to prove Theorem~ \ref{thm 3.1}, we need the following result:
   \begin {lemma}\label{le 3.2} Let    $\bigcup\limits_{j \in J}\{T_{\gamma}g_{j,p}\}_{\gamma \in \Gamma_j, p \in P_j}$ and $\bigcup\limits_{j \in J}\{T_{\gamma}h_{j,p}\}_{\gamma \in \Gamma_j, p \in P_j}$ satisfy the assumptions of Theorem~\ref{thm 3.1}, and  let the symbol $\Theta$  be their corresponding mixed dual Gramian operator. For $f \in \mathfrak{D}$, define the function $w_f: G \rightarrow \mathbb{C},\, x \mapsto \langle{\Theta T_xf, T_xf}\rangle$. Then, the following hold true:
    \begin{itemize}
    	\item [(i)] The operator $\Theta$ commutes with all translations  $T_{x}$ for $x \in G$, if, and only if,  $w_f$ is constant for all $f \in \mathfrak{D}$, that means,  $w_f(x)=w_f(0)=\langle{\Theta f, f}\rangle$, for all $x \in G$, where $0$ denotes the identity element of the LCA group $G$. 
    	\item [(ii)] Assume that, for $f \in \mathfrak{D}$, the $\alpha$-LIC holds. Then, the function $w_f(x)$ is a continuous function that coincides pointwise with its absolutely convergent $($almost periodic$)$ Fourier series 
    \begin{equation}\label{eq 3.3}
    \sum \limits_{\alpha \in \bigcup \limits_{j \in J} \Gamma_j^{\perp}}\alpha(x) \widehat{w}_f(\alpha),
    \end{equation} 
    where 
    \begin{equation}\label{eq 3.4}
    \widehat{w}_f(\alpha):=\int\limits_{\widehat{G}}\widehat{f}(\xi)\overline{\widehat{f}(\xi+ \alpha)}t_{\alpha}(\xi)d\mu_{\widehat{G}}(\xi),
    \end{equation} 
    converges absolutely.
    \item [(iii)] $ w_f$ is constant for all $f \in \mathfrak{D}$ if, and only if, for all $\alpha \in \bigcup\limits_{j \in J}\Gamma_j^{\perp}\setminus \{0\},\, t_{\alpha}(\xi)=0\, ~\mbox{a.e.}~\, \xi \in \widehat{G}.$
    \end{itemize}
\end{lemma}
  \begin{proof} 
  (i) Let $\Theta T_{x}=T_{x} \Theta$, for all $x \in G$. Then, the direct part of (i) can be concluded by observing that  
  \begin{align*} 
  	w_f(x)=  \langle{\Theta T_xf, T_xf}\rangle=\langle{ T_x\Theta f, T_xf}\rangle=\langle{\Theta f, T_x^{\ast}T_xf}\rangle=\langle{\Theta f,f}\rangle,
  	\end{align*}
  for all $x \in G$ and $f \in \mathfrak{D}$, since for each $x$, $T_{x}$ is an unitary operator.
  
   Conversely, let $w_f$ be constant for all $f \in \mathfrak{D}$. Then, for all $x \in G$, 
  	\begin{align*}
  	w_f(x)=  \langle{\Theta T_xf, T_xf}\rangle=\langle{ T_{-x}\Theta T_x f, f}\rangle=\langle{\Theta f, f}\rangle, 
  	\end{align*}
  	which by using unitary nature of $T_x$ for each $x$ and polarization identity, leads to  $T_{-x}\Theta T_x=\Theta$, and hence, we get $\Theta T_x=T_x \Theta$.
  	
  	(ii) For each $f \in \mathfrak{D}$ and $x \in G$, we can write the function  \begin{align*}
  	w_f(x)= \langle{\Theta T_xf, T_xf}\rangle &=
  	\Big\langle{\sum \limits_{j \in J}\int \limits_ {p \in P_j} \int \limits_{\Gamma_j}\big\langle{T_xf,T_{\gamma}h_{j,p}}\big \rangle T_{\gamma}g_{j,p} d\mu_{\Gamma_j}(\gamma) d\mu_{P_j}(p), T_xf }\Big\rangle\\
  	&=
  	\sum \limits_{j \in J}\int \limits_ {p \in P_j} \int \limits_{\Gamma_j}\big\langle{T_xf,T_{\gamma}h_{j,p}}\big \rangle\big\langle{T_{\gamma}g_{j,p},T_xf} \big\rangle d\mu_{\Gamma_j}(\gamma) d\mu_{P_j}(p).\tag{*}
  	\end{align*}
  	 Now, by proceeding in the same way as in the proof of \cite[Theorem 3.4]{JL}, the result follows.
  	 
  	(iii)
  	Let us assume that $\alpha$-LIC holds for all $f \in \mathfrak{D}$. From (\ref{eq 3.3}) and (*), it follows that 
  	\begin{equation}\label{eq 3.5}
  	w_f(x)= \sum \limits_{j \in J}\int \limits_ {p \in P_j} \int \limits_{\Gamma_j}\big\langle{T_xf,T_{\gamma}h_{j,p}}\big \rangle\big\langle{T_{\gamma}g_{j,p},T_xf} \big\rangle d\mu_{\Gamma_j}(\gamma) d\mu_{P_j}(p)
  	= \sum \limits_{\alpha \in \bigcup \limits_{j \in J} \Gamma_j^{\perp}}\alpha(x) \widehat{w}(\alpha).
  	\end{equation}
  	Consider now the function $z_f(x):= w_f(x)-\langle{\Theta f, f}\rangle$ which is continuous in view of continuity of the function $w_f$.  
  	
  	Now, for the direct part, assume that the function $w_f$ is constant for all $f \in \mathfrak{D}$. We claim that  $t_{\alpha}(\xi)=0$, for all $\alpha \in \bigcup\limits_{j \in J}\Gamma_j^{\perp}\setminus \{0\}$ and $a.e.$ $\xi \in \widehat{G}$. Here, note that by the construction $z_f$ is identical to the zero function. Additionally, since $w_f$ equals an absolute convergent, generalized Fourier series, also $z_f$ can be expressed as an absolute convergent generalized Fourier series  $z_f(x)=\sum \limits_{\alpha \in \bigcup \limits_{j \in J} \Gamma_j^{\perp}}\alpha(x) \widehat{z}_f(\alpha)$, with 
  	\begin{equation*}\label{eq 3.6}
  	\widehat{z}_f(\alpha) = \begin{cases} \widehat{w}_f(0)-\langle{\Theta f, f}\rangle, & \mbox{if } \alpha=0,  \\ \widehat{w}_f(\alpha), & \mbox{if } \alpha \neq 0. \end{cases}
  	\end{equation*}
  	By the uniqueness theorem for generalized Fourier series \cite[Theorem 7.12]{CC}, the function $z_f(x)$ is identical to zero if, and only, if  $\widehat{z}_f(\alpha)=0$ for all $\alpha \in \bigcup \limits_{j \in J} \Gamma_j^{\perp}.$
 
  	Thus, for $\alpha \in \bigcup \limits_{j \in J} \Gamma_j^{\perp}$ and $f \in \mathfrak{D}$, we have 
  \begin{equation}\label{eq 3.7}
  \widehat{w}_f(\alpha)=\delta_{\alpha,0} \langle{\Theta f,f}\rangle.
  \end{equation} 
  Let $\alpha \neq 0$. Then, for all $f \in \mathfrak{D}$, (\ref{eq 3.7}) reduces to $\widehat{w}_f(\alpha)=0$, and hence, we get 
  \begin{equation}\label{eq 3.8}
  \int \limits_{\widehat{G}} \widehat{f}(\xi)\overline{\widehat{f}(\xi + \alpha)} t_{\alpha}(\xi)d\mu_{\widehat{G}}(\xi)=0,\, ~\mbox{for a.e.}~\, \xi \in \widehat{G}. 
  \end{equation}
  Now, define the multiplication operator 
  $M_{\overline{t}_{\alpha}}: L^2(\widehat{G}) \rightarrow L^2(\widehat{G})$ by  $M_{\overline{t}_{\alpha}}\widehat{f}(\xi)= \overline{t_{\alpha}(\xi)} \widehat{f}(\xi)$ which is a bounded linear operator  in view of  the fact that $t_{\alpha}(\xi) \in L^{\infty}(\widehat{G})$  (for details, see  (\ref{eq 4.3})). For all $f \in \mathfrak{D}$ and a.e. $\xi \in \widehat{G} $, we can now rewrite the term in left hand side of (\ref{eq 3.8}) as 
  \begin{align*}
  \int \limits_{\widehat{G}} \widehat{f}(\xi)\overline{\overline{t_{\alpha}(\xi)} T_{\alpha}\widehat{f}(\xi) }d\mu_{\widehat{G}}(\xi)
  &= \int \limits_{\widehat{G}} \widehat{f}(\xi)\overline{M_{\overline{t}_{\alpha}} (T_{\alpha}\widehat{f})(\xi)} d\mu_{\widehat{G}}(\xi)
  = \int \limits_{\widehat{G}} \widehat{f}(\xi)\overline{(M_{\overline{t}_{\alpha}} T_{\alpha})\widehat{f}(\xi)} d\mu_{\widehat{G}}(\xi)\\
  &= {\big\langle{\widehat{f}, M_{\overline{t}_{\alpha}} T_{\alpha}\widehat{f}\,
  }\big\rangle}_{L^2(\widehat{G})}, 
  \end{align*}
   which is equal to zero in view of (\ref{eq 3.8}). From the above equality and the fact that $\mathfrak{D}$ is dense in the complex Hilbert space $L^2(G)$, it follows that 
   $M_{\overline{t}_{\alpha}} T_{\alpha}\widehat{f}=0$, which is if, and only if, $M_{\overline{t}_{\alpha}}T_{\alpha}
   =0$, that means,  $M_{\overline{t}_{\alpha}}T_{\alpha}(\widehat{g})=0$ for all $\widehat{g} \in L^2(\widehat{G})$ , and hence,  $M_{\overline{t}_{\alpha}}T_{\alpha}\widehat{g}(\xi)=\overline{t_{\alpha}(\xi)} T_{\alpha}\widehat{g}(\xi)=0$ for all $\widehat{g} \in L^2(\widehat{G})$ and a.e. $\xi \in \widehat{G}$. Thus, (\ref {eq 3.8}) holds if, and only if, for a.e. $\xi \in \widehat{G}$, we have $t_{\alpha}(\xi)=0$, for all $\alpha \in \bigcup\limits_{j \in J}\Gamma_j^{\perp}\setminus \{0\}$.
   
   Conversely, for each  $\alpha \in \bigcup\limits_{j \in J}\Gamma_j^{\perp}\setminus \{0\}$, let $t_{\alpha}(\xi)=0$ for a.e. $\xi \in \widehat{G}$, which implies that $\widehat{w}_f(\alpha)=0$, and by using this in (\ref{eq 3.5}) along with the fact from
   (\ref{eq 3.7}) that for $\alpha=0,$  we have $\widehat{w}_f(\alpha)=\langle{\Theta f,f}\rangle$, and hence,  
   \begin{equation*}\label{eq 3.9}
   w_f(x)= \sum \limits_{\alpha \,\in \bigcup\limits_{j \in J}\Gamma_j^{\perp}\setminus \{0\}}\alpha(x)\widehat{w}_f(\alpha) + \sum \limits_{\alpha\in\{0\}}\alpha(x)\widehat{w}_f(\alpha) = 0+ \widehat{w}_f(0)= \langle{\Theta f, f}\rangle, 
   \end{equation*}
 for a.e. $x \in G$. Therefore, $w_f$ is constant for all $f \in \mathfrak{D}$.
  \end{proof}  

   \begin{proof}[Proof of  Theorem~\ref{thm 3.1}]
   Clearly, part (i) is true if, and only if, (\ref{eq 3.1}) holds in view of Lemma~\ref{le 3.2}. Further, it is well known that if the mixed dual Gramian operator, say $\Theta$, commutes with $T_{x}$ for all $ x \in G$, then it is a Fourier multiplier (see \cite[Theorem 4.1.1]{L}), and hence there exists a unique $s \in L^{\infty}(\widehat{G})$ such that  $\widehat{\Theta f}(\xi)= s(\xi)\widehat{f}(\xi)$, where $s(\xi)$ represents the symbol corresponding to $\Theta$. Now, for a.e. $\xi \in \widehat{G}$, we are interested in finding the expression for $s(\xi)$. For this, observe that 
   \begin{equation}\label{eq 3.10}
   \langle{\Theta f, f}\rangle={\langle{\widehat{\Theta f}, \widehat{f}}\rangle}_{L^2(\widehat{G})}= \int \limits_{\widehat{G}} \widehat{\Theta f}(\xi)\overline{\widehat{f}(\xi)}d \mu_{\widehat{G}}(\xi)=  \int \limits_{\widehat{G}}  s(\xi)\widehat{f}(\xi)  \overline{\widehat{f}(\xi)}d \mu_{\widehat{G}}(\xi).
   \end{equation}
   Moreover, for $\alpha=0$, it follows from (\ref{eq 3.4}) and (\ref{eq 3.7}) that for all $f \in \mathfrak{D}$, 
    \begin{equation}\label{eq 3.11}
   \langle{\Theta f, f}\rangle= \widehat{w}_f(0)=\int \limits_{\widehat{G}} \widehat{f}(\xi) \overline{\widehat{f}(\xi)} \sum \limits_{j \in J} \int \limits_{P_j}\overline{\widehat{h}_{j,p}(\xi)}\widehat{g}_{j,p}(\xi)d \mu_{P_j}(p)d \mu_{\widehat{G}}(\xi).
   \end{equation}
   Therefore, since (\ref{eq 3.10}) and 
   (\ref{eq 3.11}) are valid for all $f \in \mathfrak{D}$ and $s$ is unique, it is clear that the symbol of $\Theta$, that is,  $s(\xi)=\sum \limits_{j \in J} \int \limits_{P_j}\overline{\widehat{h}_{j,p}(\xi)}\widehat{g}_{j,p}(\xi)d \mu_{P_j}(p).$
  \end{proof}
   Now, we are ready to prove our first main result, that is, Theorem~\ref{thm 2.1}, which is as follows:
   
  \begin{proof}[Proof of  Theorem~\ref{thm 2.1}] By Definition~\ref{de 1.4}, the part (i) is equivalent in saying that the mixed dual Gramian operator corresponding to the GTI-systems  $\bigcup\limits_{j \in J}\{T_{\gamma}g_p\}_{\gamma \in \Gamma_j, p \in P_j}$ and $\bigcup\limits_{j \in J}\{T_{\gamma}h_p\}_{\gamma \in \Gamma_j, p \in P_j}$, say $\Theta$, is equal to zero.
    Next, we claim that $\Theta=0$ if, and only if, $\Theta$ commutes with the translations $T_{x}$ for all $x \in G$, and, act as a Fourier multiplier with symbol  \begin{align*}
    s(\xi) =\sum\limits_{j \in J} \int \limits_{P_j} \overline{\widehat {h}_{j,p}(\xi)} \widehat{g}_{j,p}(\xi)d{\mu}_{P_j}(p)=0,\, ~\mbox{for a.e.}~\, \xi \in \widehat{G}. \end{align*}
     For proving the above claim, let $\Theta=0$. Then, $\Theta T_{x}(f)=0$, for all $x \in G$ and $f \in L^2(G)$. Since for each $x$, translation $T_{x}$ is a linear operator, therefore $T_{x}(0)=~\mbox{zero of}~ L^2(G)=0$, and hence,  $T_{x}\Theta f=T_{x}(0)=0$, which implies that $\Theta T_{x}= T_{x}\Theta$ for all $x \in G$. Thus by  \cite[Theorem 4.1.1]{L}, $\Theta$ is a Fourier multiplier. So for all $f \in L^2(G)$ we have  $0=\widehat{\Theta f}(\xi)= s(\xi)\widehat{f}(\xi)$,   $\xi \in \widehat{G}$ a.e., where $s(\xi)$, the symbol of $\Theta$ as a Fourier multiplier, is given by (\ref{eq 3.2}). 
     
     Conversely, if $\Theta$ is a Fourier multiplier with symbol $s(\xi)=0$, then $ \widehat{\Theta f}(\xi)=0$, which implies that $\Theta f=0$ for all $f \in L^2(G)$, and hence, $\Theta=0.$ Now, the result follows by considering the above claim along with Theorem~\ref{thm 3.1}.
     \end{proof}
     
  
  \section{ A characterization result for duals of super GTI-systems }\label{sec 5}
  The content in this section centers around the proof of our second main result, that is, Theorem~\ref{thm 4.51}. For this, we need to prove the following  result which is a continuous version of \cite[Theorem 3]{A}. 
  
  \begin{lemma}\label{prop 5.1}
  Let $P$ be an orthogonal projection from a complex $($separable$)$ Hilbert space $\mathcal{H}$ onto a closed subspace $\mathcal{H}_1$ in $\mathcal{H}$, and let $\{X_j\}_{j \in \mathbb{J}}$ and $\{Y_j\}_{j \in \mathbb{J}}$ be Bessel families in $\mathcal{H}$, where $($$\mathbb{J}$, $\sum_{\mathbb{J}}$, $\mu_{\mathbb{J}}$$)$ denotes a measure space with $\sum_{\mathbb{J}}$ as the $\sigma$-algebra  and $\mu_{\mathbb{J}}$ as the non-negative measure.  Then, the following assertions are true:
  \begin{itemize}
  \item [(i)] If $\{X_j\}_{j \in \mathbb{J}}$ is a continuous frame for $\mathcal{H}$, then $\{PX_j\}_{j \in \mathbb{J}}$ is a continuous frame for $\mathcal{H}_1$ with same frame bounds.
   \item [(ii)] If $\{X_j\}_{j \in \mathbb{J}}$ and $\{Y_j\}_{j \in \mathbb{J}}$ are dual frames for $\mathcal{H}$, then $\{PX_j\}_{j \in \mathbb{J}}$ and $\{PY_j\}_{j \in \mathbb{J}}$ are dual frames for $\mathcal{H}_1$.
   \end{itemize}
  \end{lemma}
  
  \begin{proof}
 For the part (i), let $\{X_j\}_{j \in \mathbb{J}}$ be a  continuous frame for $\mathcal{H}$ with frame bounds $0 < \alpha_1 \leq \alpha_2$. 
  We claim that $\{PX_j\}_{j \in \mathbb{J}}$ is a frame for $\mathcal{H}_1$. To conclude this claim, note that $j \mapsto PX_j$ is weakly measurable, that is, for all $h \in \mathcal{H}_1$, the mapping $\mathbb{J} \mapsto \mathbb{C};\,j \mapsto \langle{h, PX_j}\rangle$ is measurable. Now, to check the frame condition for an arbitrary element $h \in \mathcal{H}_1$, we use the properties of orthogonal projection $P$ such as $P^{\ast}=P$ and ${||Ph||}^2={||h||}^2$ in the left inequality of (\ref{eq 1.1}) with the frame bound $\alpha_1$ to obtain
  \begin{align*}
  \alpha_1||h||^2 = \alpha_1||Ph||^2\leq
  \int \limits_{\mathbb{J}}|\langle{Ph, X_j}\rangle|^2d\mu_{\mathbb{J}}(j)
   =\int \limits_{\mathbb{J}}|\langle{h, P^{\ast}X_j}\rangle|^2d\mu_{\mathbb{J}}(j)=\int \limits_{\mathbb{J}}|\langle{h, PX_j}\rangle|^2d\mu_{\mathbb{J}}(j),
  \end{align*}
   along with a similar estimate with the frame bound $\alpha_2$ which yields 
  \begin{align*}
  \int \limits_{\mathbb{J}}|\langle{h, PX_j}\rangle|^2d\mu_{\mathbb{J}}(j)=\int \limits_{\mathbb{J}}|\langle{P^{\ast}h, X_j}\rangle|^2d\mu_{\mathbb{J}}(j)\leq \alpha_2||P^{\ast}h||^2=\alpha_2||Ph||^2=\alpha_2||h||^2.
  \end{align*}
  
 To prove the part (ii), let $\{X_j\}_{j \in \mathbb{J}}$ and $\{Y_j\}_{j \in \mathbb{J}}$ be dual frames for $\mathcal{H}$. Our  claim is to show that $\{PX_j\}_{j \in \mathbb{J}}$ and $\{PY_j\}_{j \in \mathbb{J}}$ are dual frames for $\mathcal{H}_1$. For this, we simply write any arbitrary $h \in \mathcal{H}_1$ in terms of the continuous frame $\{X_j\}_{j \in \mathbb{J}}$ of $\mathcal{H}$, and use the  commutativity of $P$ with the integral over a general measure space  $($$\mathbb{J}$, $\sum_{\mathbb{J}}$, $\mu_{\mathbb{J}}$$)$ in the following computation:
  \begin{align*}
  h=Ph&=P\Big(\int \limits_{\mathbb{J}}\langle{h, Y_j}\rangle X_jd\mu_{\mathbb{J}}(j)\Big)=P\Big(\int \limits_{\mathbb{J}}\langle{Ph, Y_j}\rangle X_jd\mu_{\mathbb{J}}(j)\Big)\\
  &=\int \limits_{\mathbb{J}}\langle{h, P^{\ast}Y_j}\rangle PX_jd\mu_{\mathbb{J}}(j)= \int \limits_{\mathbb{J}}\langle{h, PY_j}\rangle PX_jd\mu_{\mathbb{J}}(j),
  \end{align*}
  where the interchange of $P$ with the integral is guaranteed by the fact that projections are closed and  the mapping $j \mapsto X_j$ is weakly measurable.
  \end{proof}
  The next result is a continuous version of  \cite[Theorem 7]{Ba} which plays a significant role in this sequel:
  \begin{theorem}\label{thm 4.1}
   Let $\mathcal{H}_n$ be a complex $($separable$)$ Hilbert space for $n=1,2,\ldots,N$, and let $\{x^{(n)}_j\}_{j \in \mathbb{J}}$ and $\{y^{(n)}_j\}_{j \in \mathbb{J}}$  be  Bessel families in $\mathcal{H}_n$ for each $n$, where $($$\mathbb{J}$, $\sum_{\mathbb{J}}$, $\mu_{\mathbb{J}}$$)$ denotes a measure space with $\sum_{\mathbb{J}}$ as the $\sigma$-algebra  and $\mu_{\mathbb{J}}$ as the non-negative measure. Then, the families $\big\{ \bigoplus \limits_{n=1}^{N} x^{(n)}_j\big\}_{j \in \mathbb{J}}:=\{x^{(1)}_j \oplus \cdots \oplus x^{(N)}_j\}_{j \in \mathbb{J}}$ and $\big\{ \bigoplus \limits_{n=1}^{N} y^{(n)}_j\big\}_{j \in \mathbb{J}}:=\{y^{(1)}_j \oplus \cdots \oplus y^{(N)}_j\}_{j \in \mathbb{J}}$ are dual frames for $\bigoplus \limits_{n=1}^{N} \mathcal{H}_n:=\mathcal{H}_1 \oplus \cdots \oplus \mathcal{H}_n$ if, and only if, both of the following two conditions hold:
  	\begin{itemize}
  		\item [(i)] for each $n$, $\{x^{(n)}_j\}_{j \in \mathbb{J}}$ and $\{y^{(n)}_j\}_{j \in \mathbb{J}}$ are dual frames for $\mathcal{H}_n$,
  		\item [(ii)]  for $n_1,n_2=1,2,\ldots,N$ with $n_1 \neq n_2$, $\{x^{(n)}_j\}_{j \in \mathbb{J}}$ and $\{y^{(n)}_j\}_{j \in \mathbb{J}}$
  		are orthogonal frames.
  	\end{itemize} 
  	In particular, for each  $1\leq n \leq  N$ and $j \in J$,  by using  $x^{(n)}_j=y^{(n)}_j$  in the above result, we get a characterization of $\{x^{(n)}_j\}_{j \in \mathbb{J}}$ such that  $\big\{ \bigoplus \limits_{n=1}^{N} x^{(n)}_j\big\}_{j \in \mathbb{J}}$ forms a Parseval frame for $\bigoplus \limits_{n=1}^{N} \mathcal{H}_n$.
  \end{theorem}
  
   \begin{proof}
   	 For each coordinate $n$ of $\mathcal{H}_1\oplus \cdots \oplus\mathcal{H}_N$, let by $P_n$, we denote the orthogonal projection onto $0\oplus \cdots \oplus \mathcal{H}_n \oplus \cdots \oplus 0$. Let the families $\big\{ \bigoplus \limits_{n=1}^{N} x^{(n)}_j\big\}_{j \in \mathbb{J}}$ and  $\big\{ \bigoplus \limits_{n=1}^{N} y^{(n)}_j\big\}_{j \in \mathbb{J}}$ be dual frames for $\bigoplus \limits_{n=1}^{N} \mathcal{H}_n$. Then, Lemma~\ref{prop 5.1} says that for each $n$, $\{x^{(n)}_j\}_{j \in \mathbb{J}}= P_n\Big(\big\{ \bigoplus \limits_{n=1}^{N} x^{(n)}_j\big\}_{j \in \mathbb{J}}\Big)$ and  
   	$\{y^{(n)}_j\}_{j \in \mathbb{J}}= P_n\Big(\big\{ \bigoplus \limits_{n=1}^{N} y^{(n)}_j\big\}_{j \in \mathbb{J}}\Big)$ are dual frames for $\mathcal{H}_n$.
 	Moreover, if $n_1,n_2=1,2,\ldots,N$ with $n_1 \neq n_2$, then for every $h \in \bigoplus \limits_{n=1}^{N} \mathcal{H}_n$, by using the properties of orthogonal projection operator, we can write the following:
   	 \begin{align*}
  	0&=P_{n_1}\big(P_{n_2}h\big)= P_{n_1}\Bigg(\int \limits_{\mathbb{J}}\Big\langle{P_{n_2}h,  \bigoplus \limits_{n=1}^{N} x^{(n)}_j}\Big\rangle\bigoplus \limits_{n=1}^{N} y^{(n)}_jd\mu_{\mathbb{J}}(j)\Bigg)
   	=
   	P_{n_1}\Bigg(\int \limits_{\mathbb{J}}\Big\langle{P^2_{n_2}h,  \bigoplus \limits_{n=1}^{N} x^{(n)}_j}\Big\rangle\bigoplus \limits_{n=1}^{N} y^{(n)}_jd\mu_{\mathbb{J}}(j)\Bigg)\\
   	&=
   	P_{n_1}\Bigg(\int \limits_{\mathbb{J}}\Big\langle{P_{n_2}h, P^{\ast}_{n_2}\Big(\bigoplus \limits_{n=1}^{N} x^{(n)}_j\Big)}\Big\rangle \bigoplus \limits_{n=1}^{N} y^{(n)}_jd\mu_{\mathbb{J}}(j)\Bigg)
   	=
   	\int \limits_{\mathbb{J}}\Big\langle{P_{n_2}h,  P_{n_2}\Big(\bigoplus \limits_{n=1}^{N} x^{(n)}_j\Big)}\Big\rangle P_{n_1}\Big(\bigoplus \limits_{n=1}^{N} y^{(n)}_j\Big)d\mu_{\mathbb{J}}(j)\\
   	&=\int \limits_{\mathbb{J}}\Big\langle{P_{n_2}h,  \big(0\oplus \cdots \oplus x^{(n_2)}_j\oplus \cdots \oplus 0\big)}\Big\rangle\big(0\oplus \cdots \oplus y^{(n_1)}_j\oplus \cdots \oplus 0\big) d\mu_{\mathbb{J}}(j),
   	\end{align*}
   	and hence, we obtain
   	\begin{align*}
   	0&= \int \limits_{\mathbb{J}}\big\langle{P_{n_2}h, x^{(n_2)}_j}\big\rangle\big(0\oplus \cdots \oplus y^{(n_1)}_j\oplus \cdots \oplus 0\big) d\mu_{\mathbb{J}}(j)\\
   	&= 0\oplus \cdots \oplus \Big(\int \limits_{\mathbb{J}}\big\langle{P_{n_2}h, x^{(n_2)}_j}\big\rangle y^{(n_1)}_j d\mu_{\mathbb{J}}(j)\Big)\oplus \cdots \oplus 0, 
   	\end{align*}
   	which implies that 
   	$\displaystyle\int \limits_{\mathbb{J}}\big\langle{P_{n_2}h, x^{(n_2)}_j}\big\rangle y^{(n_1)}_j d\mu_{\mathbb{J}}(j)=0$, that means, $\displaystyle\int \limits_{\mathbb{J}}\big\langle{\widetilde{h}, x^{(n_2)}_j}\big\rangle y^{(n_1)}_j d\mu_{\mathbb{J}}(j)=0$, for all 
$\widetilde{h} \in \mathcal{H}_{n_2}$, and hence 
$\{x^{(n_2)}_j\}_{j \in \mathbb{J}}$ and  $\{y^{(n_2)}_j\}_{j \in \mathbb{J}}$ are orthogonal frames.
Conversely, let us assume that both the conditions (i) and (ii) hold true. Then, for every $h \in \bigoplus \limits_{n=1}^{N} \mathcal{H}_n$, we can write
\begin{align*}
\int \limits_{\mathbb{J}}\Big\langle{h,  \bigoplus \limits_{n=1}^{N} x^{(n)}_j}\Big\rangle\bigoplus \limits_{n=1}^{N} y^{(n)}_jd\mu_{\mathbb{J}}(j)&=\int \limits_{\mathbb{J}}\Big\langle{\bigoplus \limits_{n=1}^{N} P_nh,  \bigoplus \limits_{n=1}^{N} x^{(n)}_j}\Big\rangle\bigoplus \limits_{n=1}^{N} y^{(n)}_jd\mu_{\mathbb{J}}(j)
=
\int \limits_{\mathbb{J}}\sum \limits_{n=1}^{N} \Big\langle{ P_nh,  x^{(n)}_j}\Big\rangle\bigoplus \limits_{n=1}^{N} y^{(n)}_jd\mu_{\mathbb{J}}(j)\\
&=
\Big(\int \limits_{\mathbb{J}}\sum \limits_{n=1}^{N} \big\langle{ P_nh,  x^{(n)}_j}\big\rangle y^{(1)}_jd\mu_{\mathbb{J}}(j)\Big)
\oplus \cdots \oplus
\Big(\int \limits_{\mathbb{J}}\sum \limits_{n=1}^{N} \big\langle{ P_nh,  x^{(n)}_j}\big\rangle y^{(N)}_jd\mu_{\mathbb{J}}(j)\Big)\\
&= \Big(\int \limits_{\mathbb{J}}\big\langle{ P_1h,  x^{(1)}_j}\big\rangle y^{(1)}_jd\mu_{\mathbb{J}}(j)\Big)
\oplus \cdots \oplus
\Big(\int \limits_{\mathbb{J}} \big\langle{ P_Nh,  x^{(N)}_j}\big\rangle y^{(N)}_jd\mu_{\mathbb{J}}(j)\Big)\\
&= P_1h\oplus \cdots \oplus P_Nh=h,
\end{align*}
and hence, we conclude that
\begin{equation}\label{eq 5.1}
\langle{h,h_1}\rangle=\int \limits_{\mathbb{J}}\Big\langle{h,  \bigoplus \limits_{n=1}^{N} x^{(n)}_j}\Big\rangle \Big\langle{\bigoplus \limits_{n=1}^{N} y^{(n)}_j, h_1}\Big\rangle d\mu_{\mathbb{J}}(j),\,\,~\mbox{for all}~\,h, h_1 \in \bigoplus \limits_{n=1}^{N} \mathcal{H}_n.
\end{equation}
Next, we claim that  if for each $1\leq n \leq N$, $\{x^{(n)}_j\}_{j \in \mathbb{J}}$ is a  Bessel family in $\mathcal{H}_n$ with Bessel constant ${\alpha}^{(n)}_2$ (say), then  $\big\{ \bigoplus \limits_{n=1}^{N} x^{(n)}_j\big\}_{j \in \mathbb{J}}$  satisfies the Bessel condition due to the following computation:
\begin{align*}
\int \limits_{\mathbb{J}}\Big|\Big\langle{h,  \bigoplus \limits_{n=1}^{N} x^{(n)}_j}\Big\rangle \Big|^2d\mu_{\mathbb{J}}(j)
&=\int \limits_{\mathbb{J}}\Big|\Big\langle{\bigoplus \limits_{n=1}^{N}P_nh,  \bigoplus \limits_{n=1}^{N} x^{(n)}_j}\Big\rangle \Big|^2d\mu_{\mathbb{J}}(j)
= \int \limits_{\mathbb{J}}\Big|\sum \limits_{n=1}^{N}\big\langle{P_nh,  x^{(n)}_j}\big\rangle \Big|^2d\mu_{\mathbb{J}}(j)\\
& \leq \int \limits_{\mathbb{J}}\sum \limits_{n=1}^{N}\big|\big\langle{P_nh,  x^{(n)}_j}\big\rangle \big|^2d\mu_{\mathbb{J}}(j)
\leq  \sum \limits_{n=1}^{N}{\alpha}^{(n)}_2||h||^2 ,\,\,~\mbox{for all}~\,h \in \bigoplus \limits_{n=1}^{N} \mathcal{H}_n.
\end{align*}
Hence the result follows in view of the fact that  $\big\{ \bigoplus \limits_{n=1}^{N} x^{(n)}_j\big\}_{j \in \mathbb{J}}$ and $\big\{ \bigoplus \limits_{n=1}^{N} y^{(n)}_j\big\}_{j \in \mathbb{J}}$ are Bessel families satisfying the  inequality (\ref{eq 5.1}).  
   \end{proof}
  The following result is an easy consequence of Theorem~\ref{thm 4.1},  the proof for which follows by replacing the general Hilbert spaces  $\mathcal{H}_n$, and the sequences $\{x^{(n)}_j\}_{j \in \mathbb{J}}$ and  $\{y^{(n)}_j\}_{j \in \mathbb{J}}$ in Theorem~\ref{thm 4.1} respectively with $L^2(G)$, and the GTI systems 
   $\bigcup\limits_{j \in J}\{ T_{\gamma}g^{(n)}_{j,p}
  		\}_{\gamma \in \Gamma_j,\, p \in P_j}$
  and 
   $\bigcup\limits_{j \in J}\{ T_{\gamma}h^{(n)}_{j,p}
  		\}_{\gamma \in \Gamma_j,\, p \in P_j}$ for each $n=1,2,\ldots, N$: 
  \begin{corollary}\label{prop 4.3}
  	For each $1 \leq n \leq N$, let the GTI systems  $\bigcup\limits_{j \in J}\{ T_{\gamma}g^{(n)}_{j,p}
  	\}_{\gamma \in \Gamma_j,\, p \in P_j}$ and $\bigcup\limits_{j \in J}\{ T_{\gamma}h^{(n)}_{j,p}
  	\}_{\gamma \in \Gamma_j,\, p \in P_j}$ be Bessel families in $L^2(G)$. Then, $\big\{\bigoplus \limits_{n=1}^{N} g^{(n)}_{j,p}\big\}_{p \in P_j,\,j \in J}$ and $\big\{\bigoplus \limits_{n=1}^{N} h^{(n)}_{j,p}\big\}_{p \in P_j,\,j \in J}$ form a super-dual frame pair in $L^2(G)^{(N)}$ if, and only if, both of the following  hold:
  	\begin{itemize}
  		\item [(i)] for each $1 \leq n \leq N$, $\{g^{(n)}_{j,p}\}_{p \in P_j,\,j \in J}$ and $\{h^{(n)}_{j,p}\}_{p \in P_j,\,j \in J}$ form a dual frame pair in  $L^2(G)$,
  		\item [(ii)] 
  		for $n_1,n_2=1,2,\ldots,N$ with $n_1 \neq n_2$, $\bigcup\limits_{j \in J}\{T_{\gamma}g^{(n_1)}_{j,p}\}_{\gamma \in \Gamma_j,\, p \in P_j}$ and $\bigcup\limits_{j \in J}\{T_{\gamma}h^{(n_2)}_{j,p}\}_{\gamma \in \Gamma_j,\, p \in P_j}$ are GTI-orthogonal frame systems in $L^2(G)$. 
  	\end{itemize}
  		In particular, 
  		for each  $1\leq n \leq  N$, $j \in J$ and $p \in P_j$, by using $g^{(n)}_{j,p}= h^{(n)}_{j,p}$ in
  		 the above result, we get a characterization of $\big\{\bigoplus \limits_{n=1}^{N} g^{(n)}_{j,p}\big\}_{p \in P_j,\,j \in J}$ such that  the super-GTI system generated by $\big\{\bigoplus \limits_{n=1}^{N} g^{(n)}_{j,p}\big\}_{p \in P_j,\,j \in J}$  forms a Parseval frame for $L^2(G)^{(N)}$. 
  \end{corollary}
  	 
  \begin{proof}[Proof of  Theorem~\ref{thm 4.51}]
  Observe that Corollary~\ref{prop 4.3}(i) is equivalent to (\ref{eq 2.51}) in view of a result on dual frames from \cite[Theorem 3.4]{JL}, and hence the proof follows by using this fact in Corollary~\ref{prop 4.3} along with Theorem~\ref{thm 3.1}, that is, a characterization result for GTI-orthogonal frame systems. 
  \end{proof}
  \section{Applications of the main characterization results}\label{sec 6}
   The purpose of this section is to discuss applications of our main results stated in Subsection~\ref{sec 3.3}, that is, Theorem~\ref{thm 2.1} and Theorem~\ref{thm 4.51} to the Bessel families having wave-packet structure which are obtained by applying certain collections of dilations, modulations and translations to a countable family of functions in $L^2(G)$. As a consequence, we obtain results for wavelet and Gabor systems in Subsection~\ref{sec 6.2}. Along with this, we connect the already existing results from the literature with the theory discussed in this article by providing various examples in case of  $G=\mathbb{R}^d,\,\mathbb{Z}^d$ etc. 
  \subsection{Wave-Packet Systems}\label{sec 6.1}
  
  \noindent
  
  For a given second countable LCA group $G$, let {{\bf{Epi}}}($G$), {{\bf{Epick}}($G$) and {{\bf{Aut}}($G$)}} respectively denote  the semigroup of continuous group homomorphisms $\alpha$ from $G$ onto $G$, the semigroup of $\alpha \in {{\bf{Epi}}}(G)$ having compact kernel $\ker \alpha$, and the group of topological automorphisms $\alpha$ of $G$ onto itself. Note that {{\bf{Aut}}($G$) $\subset$} {{\bf{Epick}}($G$) $\subset$ {{\bf{Epi}}}($G$)}. For  $\alpha$ $\in$ {{\bf{Epick}}($G$)}, we define the isometric {\textit{dilation operator}} $D_{\alpha}$ by
  \begin{align*}
  	D_{\alpha}:L^2(G)\rightarrow L^2(G);\,\,  D_{\alpha}f(x)=({\Delta(\alpha)})^{-1/2}f(\alpha(x)), \,\,\mbox{for all}~\,x \in G,
  \end{align*}
  where the {\textit{modular function}}  $\Delta:$~{{\bf{Epick}}($G$)$\rightarrow (0,\infty)$} is a semigroup homomorphism such that 
  \begin{align*}
  	\int \limits_{G}(g \circ \alpha)(x)d\mu_{G}(x)=\Delta(\alpha)\int\limits_{G} g(x) d\mu_{G}(x)
  \end{align*}
  for all integrable functions $g$ on $G$ with respect to the Haar measure $\mu_{G}$ (see \cite[Theorem 6.2]{BR}). For a character $\chi$ in $\widehat{G}$, we define the {\textit{modulation operator}} $M_{\chi}$ on $L^2(G)$ as 
  	 \begin{align*}
  	M_{\chi}(f)(x)=\chi(x)f(x),\,\,\mbox{for all}~~ x \in G,
  	\end{align*}
  and observe that for each $\chi \in \widehat{G}$, it is associated with the translation operator on $L^2(\widehat{G})$ by the relation \begin{equation}\label{eq 6.1}
  (\widehat{M_{\chi}f})(\xi)
  = \int \limits_{G}  \chi(x)f(x)\overline{\xi(x)} d\mu_{G}(x)
  =\int \limits_{G}  f(x)\overline{\big(\xi-\chi\big)(x)} d\mu_{G}(x)
  =\widehat{f}\big(\xi-\chi \big)
  =T_{\chi}\widehat{f}(\xi),
  \end{equation}
for all $f \in L^2(G)$ and a.e. $\xi \in \widehat{G}$. Further, note that for each 
$\alpha$ $\in$ {{\bf{Epick}}($G$)}, the dilation operator on $L^2(G)$ satisfies the following relation (see \cite[Lemma 6.6]{BR}):

\begin{equation}\label{eq 6.222}
\widehat{(D_{\alpha}f)}(\chi) = \begin{cases}
 (\Delta(\alpha))^{1/2}\widehat{f}({\beta}^{-1}(\chi)) & \text{for } \chi \in \beta(\widehat{G})=(\ker \alpha)^{\perp},\\
 0 & \text{otherwise,} 
 \end{cases}
 \end{equation}
for all $f \in L^2(G)$, where by $\beta:={\alpha}^{\ast}$, we denote the adjoint of $\alpha$ $ \in$ {{\bf{Epick}}($G$)} which  is a topological isomorphism  $\beta: \widehat{G} \rightarrow (\ker \alpha)^{\perp};\,\chi \mapsto \chi \circ \alpha$ in view of \cite[Proposition 6.5]{BR}. 

  Let $\mathcal{A}$ be a subset of 
  {{\bf{Epick}}($G$)}, let $\Gamma$ and  $\Lambda$ be respectively co-compact subgroups of $G$ and $\widehat{G}$, and for some index set $J \subset \mathbb{Z}$, let $\Psi:=\{\psi_j:j \in J \}$ be a subset of $L^2(G)$. Then, we define the {\textit{wave-packet system}} generated by $\Psi$ as: 
  	\begin{equation}\label{eq 6.2}
  	\mathcal{W}(\Psi, \mathcal{A}, \Gamma, \Lambda):=\{D_{\alpha}T_{\gamma}M_{\chi}\psi_j: \alpha \in \mathcal{A},\gamma \in \Gamma, \chi \in \Lambda, j \in J\}.
  	\end{equation}
  In the case of $L^2(\mathbb{R})$ and $L^2(\mathbb{R}^d)$, the systems of the above form  have been studied by several authors, including \cite{LWW,HLWW,CR}, and various references within. The wave-packet systems were originally introduced by C\'ordoba and Fefferman \cite{CF}, and the collection defined in (\ref{eq 6.2}) generalizes the notion of such systems in the context of LCA groups.
In particular, the wavelet and Gabor systems can be seen as special cases of (\ref{eq 6.2}) which we shall discuss in Subsection~\ref{sec 6.2}. 

 The following commutator relation helps in  representing the collection (\ref{eq 6.2}) in the form of a GTI system. This relation says that for each $\alpha \in \mathcal{A}$, $\gamma \in \Gamma$, $\chi \in \Lambda$, and $j \in J$, we have: 
 \begin{align*} 
 D_{\alpha}T_{\gamma}M_{\chi}\psi_j(x)&={(\Delta(\alpha))}^{-1/2}T_{\gamma}M_{\chi}\psi_j(\alpha(x))={(\Delta(\alpha))}^{-1/2}M_{\chi}\psi_j(\alpha(x)-\gamma)\\
 &={(\Delta(\alpha))}^{-1/2}M_{\chi}\psi_j(\alpha(x-\gamma_1))
 =D_{\alpha}M_{\chi}\psi_j(x-\gamma_1)=T_{\gamma_1}D_{\alpha}M_{\chi}\psi_j(x),
 \end{align*}
for all $x \in G$, and for some $\gamma_1 \in {\alpha}^{-1}\Gamma$ such that $\alpha(\gamma_1)=\gamma$.

 In the rest of this section, let $\mathcal{A}$ be a countable subset of {{\bf{Epick}}($G$)}. Then, by using the above commutator relation, the wave-packet system  $\mathcal{W}(\Psi, \mathcal{A}, \Gamma, \Lambda)$ will represent  a GTI system of the form  $\bigcup\limits_{\alpha \in \mathcal{A}}\{T_{\gamma}g_{\alpha,p}\}_{\gamma \in \Gamma_{\alpha},\, p \in P_{\alpha}}$ for
 $\Gamma_{\alpha}:={\alpha}^{-1}\Gamma$ with $\alpha \in \mathcal{A}$,
  $g_{\alpha, p}=g_{\alpha, (j,\chi)}=D_{\alpha}M_{\chi}\psi_j$ for  $(\alpha,p)=(\alpha,(j,\chi))$ in   $\mathcal{A}\times (J\times \Lambda)$. In this case, for each $\alpha \in \mathcal{A}$, the measure space $P_{\alpha}:=\{(j,\chi): j \in J,\,\chi \in \Lambda\}$ is equipped with the measure $\mu_{P_{\alpha}}:=\mu_{J \times \Lambda}={(\Delta(\alpha))}^{-1}(\mu_{J}\otimes\mu_{\Lambda})$, where the quantity ${(\Delta(\alpha))}^{-1}$ helps in avoiding the scaling factor in the calculations and $\mu_{J}$ represents the counting measure on $J$. Clearly, the measure  $\mu_{P_{\alpha}}$ is $\sigma$-finite. Here, note that  $\Gamma_{\alpha}={\alpha}^{-1}\Gamma$ is a closed co-compact subgroup of $G$ for each $\alpha \in \mathcal{A}$, in view of \cite[Proposition 6.4]{BR} and the fact that $\alpha$ is a continuous group homomorphism from $G$ onto $G$ along with $\Gamma$ as a closed subgroup of $G$.
  
    Next, we apply Theorem~\ref{thm 4.51} to the wave-packet systems  $\mathcal{W}(\Psi^{(n)}, \mathcal{A}, \Gamma, \Lambda)$ and  $\mathcal{W}(\Phi^{(n)}, \mathcal{A}, \Gamma, \Lambda)$, where for each $1\leq n \leq N$ and any index set $J \subset \mathbb{Z}$, $\Psi^{(n)}:=\{\psi^{(n)}_j\}_{j \in J}$ and $\Phi^{(n)}:=\{\varphi^{(n)}_j\}_{j \in J}$ are subsets in $L^2(G)$.
 Further, we simplify (\ref{eq 2.51}) by considering $\mathcal{W}(\Psi^{(n)}, \mathcal{A}, \Gamma, \Lambda)$ and   $\mathcal{W}(\Phi^{(n)}, \mathcal{A}, \Gamma, \Lambda)$ respectively as GTI systems $\bigcup\limits_{\alpha \in \mathcal{A}}\{T_{\gamma}g^{(n)}_{\alpha,p}\}_{\gamma \in \Gamma_{\alpha},\, p \in P_{\alpha}}$ and $\bigcup\limits_{\alpha \in \mathcal{A}}\{T_{\gamma}h^{(n)}_{\alpha,p}\}_{\gamma \in \Gamma_{\alpha},\, p \in P_{\alpha}}$ for each $1 \leq n \leq N$, where 
   $g^{(n)}_{\alpha, p}=g^{(n)}_{\alpha, (j,\chi)}=D_{\alpha}M_{\chi}\psi^{(n)}_j$ and $h^{(n)}_{\alpha, p}=h^{(n)}_{\alpha, (j,\chi)}=D_{\alpha}M_{\chi}{\varphi}^{(n)}_j$ for  $(\alpha,p)=(\alpha,(j,\chi)) \in  \mathcal{A}\times P_{\alpha}= \mathcal{A}\times (J\times \Lambda)$. Hence, for each $1 \leq n \leq N$, $\widetilde{\alpha} \in \bigcup\limits_{\alpha \in \mathcal{A}}\Gamma_{\alpha}^{\perp}$ and for a.e. $\xi \in \bigcup _{\alpha \in \mathcal{A}}(\ker \alpha)^{\perp}$, the expression (\ref{eq 2.51}) takes the following form in view of (\ref{eq 6.1}) and (\ref{eq 6.222}) along with $\beta={\alpha}^{\ast}$:
  \begin{align*}
 \mathcal{T}_{(n,\widetilde{\alpha})}(\xi)&:= \sum\limits_{\alpha \in \mathcal{A}:\, \widetilde{\alpha} \in \Gamma_{\alpha}^{\perp} } \int \limits_{P_{\alpha}} \overline{\widehat {h}_{\alpha,p}^{(n)}(\xi)} \widehat{g}_{\alpha,p}^{(n)}(\xi+  \widetilde{\alpha})d{\mu}_{P_{\alpha}}(p)\\
  &=\sum\limits_{\alpha \in \mathcal{A}:\, \widetilde{\alpha} \in \Gamma_{\alpha}^{\perp} } \int \limits_{J\times \Lambda} \overline{\widehat {h}_{\alpha,(j,\chi)}^{(n)}(\xi)} \widehat{g}_{\alpha,(j,\chi)}^{(n)}(\xi+  \widetilde{\alpha})d{\mu}_{J \times \Lambda}((j,\chi))\\
  &=\sum\limits_{\alpha \in \mathcal{A}:\, \widetilde{\alpha} \in ({\alpha^{-1}\Gamma})^{\perp}}\sum \limits_{j \in J}\int \limits_{\Lambda}
   \overline{\widehat{(D_{\alpha}M_{\chi}\varphi^{(n)}_j})(\xi)} \widehat {(D_{\alpha}M_{\chi}\psi^{(n)}_j)}(\xi+  \widetilde{\alpha})\frac{1}{\Delta(\alpha)}{d{\mu}_{\Lambda}(\chi)}\\
  &= \sum\limits_{\alpha \in \mathcal{A}:\, \widetilde{\alpha} \in {{\beta}\Gamma}^{\perp}}\sum \limits_{j \in J}\int \limits_{\Lambda}
   \overline{\widehat {M_{\chi}\varphi^{(n)}_j}({\beta}^{-1}\xi)} \widehat {M_{\chi}\psi^{(n)}_j}({\beta}^{-1}(\xi+ \widetilde{\alpha}))d{\mu}_{\Lambda}(\chi)\\
  &= \sum\limits_{\alpha \in \mathcal{A}:\, \widetilde{\alpha} \in {{\beta}\Gamma}^{\perp}}\sum \limits_{j \in J}\int \limits_{\Lambda} \overline{T_{\chi}\widehat {\varphi}^{(n)}_j({\beta}^{-1}\xi)} T_{\chi}\widehat {\psi}^{(n)}_j({\beta}^{-1}(\xi+ \widetilde{\alpha}))d{\mu}_{\Lambda}(\chi)\\
  &=\sum\limits_{\alpha \in \mathcal{A}:\, \widetilde{\alpha} \in {{\beta}\Gamma}^{\perp}}\sum \limits_{j \in J}\int \limits_{\Lambda} \overline{\widehat {\varphi}^{(n)}_j({\beta}^{-1}\xi-\chi)} \widehat {\psi}^{(n)}_j({\beta}^{-1}(\xi+ \widetilde{\alpha})-\chi)d{\mu}_{\Lambda}(\chi)=: \widetilde{\mathcal{T}}_{(n,\widetilde{\alpha})}(\xi) ~\mbox{(say)},
  \end{align*}
  whereas for the case of $\xi \in \widehat{G}\setminus \bigcup _{\alpha \in \mathcal{A}}(\ker \alpha)^{\perp}$ a.e., we get $\mathcal{T}_{(n,\widetilde{\alpha})}(\xi)=0,$ by proceeding in the similar way as above. Hence, we can write
  \begin{equation}\label{eq 6.41}
  \mathcal{T}_{(n,\widetilde{\alpha})}(\xi) = \begin{cases}
  \widetilde{\mathcal{T}}_{(n,\widetilde{\alpha})}(\xi) & \text{for a.e.}\,\,\,\, \xi \in \bigcup _{\alpha \in \mathcal{A}}(\ker \alpha)^{\perp},\\
  0 & \text{otherwise.} 
  \end{cases}
  \end{equation}

Now, to apply Theorem~\ref{thm 4.51} on the wave-packet systems, we require that for each $1 \leq n \leq N$ and for a.e. $\xi \in \widehat{G}$, $\mathcal{T}_{(n,\widetilde{\alpha})}(\xi)$ in (\ref{eq 6.41})
should be equal to $\delta_{\widetilde{\alpha},0}$ for all $\widetilde{\alpha} \in \bigcup\limits_{\alpha \in \mathcal{A}}\Gamma_{\alpha}^{\perp}$, which is not true whenever $\xi$ is an element of $\widehat{G}\setminus \bigcup _{\alpha \in \mathcal{A}}(\ker \alpha)^{\perp}$ since in this case for $\widetilde{\alpha}=0$ we have $\mathcal{T}_{(n,0)}(\xi)=0\neq 
\delta_{0,0}$ for a.e. $\xi$. But, if we assume $\alpha \in$ {{\bf{Aut}}($G$) $\subset$} {{\bf{Epick}}($G$), then 
 $(\ker \alpha)^{\perp}= (0)^{\perp}=\widehat{G}$, and hence for all $\widetilde{\alpha} \in \bigcup\limits_{\alpha \in \mathcal{A}}\Gamma_{\alpha}^{\perp}$, 
 \begin{equation}\label{eq 6.5}
 {\mathcal{T}}_{(n,\widetilde{\alpha})}(\xi)=\widetilde{\mathcal{T}}_{(n,\widetilde{\alpha})}(\xi)= \delta_{\widetilde{\alpha},0},\,\, ~\mbox{for a.e.}~\,\, \xi \in \widehat{G},
 \end{equation}
 and, in the similar way, for each $1 \leq n_1, n_2 \leq N$ and  $\widetilde{\alpha} \in \bigcup\limits_{\alpha \in \mathcal{A}}\Gamma_{\alpha}^{\perp}$, we have
 \begin{align*}
 \mathcal{T}_{(n_1,n_2,\widetilde{\alpha})}(\xi)&:= \sum\limits_{\alpha \in \mathcal{A}:\, \widetilde{\alpha} \in \Gamma_{\alpha}^{\perp} } \int \limits_{P_{\alpha}} \overline{\widehat {h}_{\alpha,p}^{(n_1)}(\xi)} \widehat{g}_{\alpha,p}^{(n_2)}(\xi+  \widetilde{\alpha})d{\mu}_{P_{\alpha}}(p)\\ &=\sum\limits_{\alpha \in \mathcal{A}:\, \widetilde{\alpha} \in {{\beta}\Gamma}^{\perp}}\sum \limits_{j \in J}\int \limits_{\Lambda} \overline{\widehat {\varphi}^{(n_1)}_j({\beta}^{-1}\xi-\chi)} \widehat {\psi}^{(n_2)}_j({\beta}^{-1}(\xi+ \widetilde{\alpha})-\chi)d{\mu}_{\Lambda}(\chi),
 \end{align*} 
 which by applying Theorem~\ref{thm 4.51} to the wave-packet systems implies that 
 \begin{equation}
 \label{eq 6.6}
 \mathcal{T}_{(n_1,n_2,\widetilde{\alpha})}(\xi)=0, \,\,~\mbox{for each}~\,\, 1\leq n_1\neq n_2 \leq N, \,\,~\mbox{and a.e.}~\,\, \xi \in \widehat{G}.
 \end{equation}
 
    The above discussion leads to the following result which provides the conditions on {$\Psi^{(n)}$ and  $\Phi^{(n)}$ such that the wave-packet systems generated by $\bigoplus \limits_{n=1}^{N}\Psi^{(n)}$ and  $\bigoplus \limits_{n=1}^{N}\Phi^{(n)}$ form dual frames in $L^2(G)^{(N)}$:
     	
     \begin{theorem}
     	\label{thm 6.11}
     	For each $1 \leq n \leq N$, let the wave-packet systems $\mathcal{W}(\Psi^{(n)}, \mathcal{A}, \Gamma, \Lambda)$ and  $\mathcal{W}(\Phi^{(n)}, \mathcal{A}, \Gamma, \Lambda)$ be Bessel families in $L^2(G)$ satisfying the corresponding dual $\alpha$-LIC, where $\mathcal{A}$ is a countable subset of {{\bf{Aut}}$($$G$$)$}. Then, the wave-packet systems generated by  $\bigoplus \limits_{n=1}^{N}\Psi^{(n)}$ and  $\bigoplus \limits_{n=1}^{N}\Phi^{(n)}$ 
     	$($we call as {\textit{super wave-packet systems}}$) $  form  dual  frames for $L^2(G)^{(N)}$ if, and only if, for a.e. $\xi \in \widehat{G}$, both of the following hold:
     	\begin{itemize}	
     		\item[(i)] 
     		for each $1 \leq n \leq N$ and  $\widetilde{\alpha} \in \bigcup \limits_{\alpha \in \mathcal{A}} \Gamma^{\perp}_{\alpha}$, we have 
     		\begin{equation}\label{eq 6.4111} 
     		\sum\limits_{\alpha \in \mathcal{A}:\, \widetilde{\alpha} \in {\Gamma}^{\perp}_{\alpha}}\sum \limits_{j \in J}\int \limits_{\Lambda} \overline{\widehat {\varphi}^{(n)}_j({\beta}^{-1}\xi-\chi)} \widehat {\psi}^{(n)}_j({\beta}^{-1}(\xi+ \widetilde{\alpha})-\chi)d{\mu}_{\Lambda}(\chi)=\delta_{\widetilde{\alpha},0}, 
     		\end{equation}
     		\item[(ii)] for each $1 \leq n_1\neq n_2 \leq N$ and  $\widetilde{\alpha} \in \bigcup \limits_{\alpha \in \mathcal{A}} {\Gamma}^{\perp}_{\alpha}$, we have 
     		\begin{equation}\label{eq 6.41111} 
     		\sum\limits_{\alpha \in \mathcal{A}:\, \widetilde{\alpha} \in {\Gamma}^{\perp}_{\alpha}}\sum \limits_{j \in J}\int \limits_{\Lambda} \overline{\widehat {\varphi}^{(n_1)}_j({\beta}^{-1}\xi-\chi)} \widehat {\psi}^{(n_2)}_j({\beta}^{-1}(\xi+ \widetilde{\alpha})-\chi)d{\mu}_{\Lambda}(\chi)=0,
     		\end{equation}
     	\end{itemize}
    where for $\beta={\alpha}^
    {\ast}$, $\Gamma^{\perp}_{\alpha}$ is given by $\beta\Gamma^{\perp}$. 	
     \end{theorem}
    
    \begin{proof}
   The proof can be concluded by  observing that for each $1 \leq n \leq N$, if we consider $\mathcal{W}(\Psi^{(n)}, \mathcal{A}, \Gamma, \Lambda)$ and  $\mathcal{W}(\Phi^{(n)}, \mathcal{A}, \Gamma, \Lambda)$ as Bessel families satisfying  corresponding dual $\alpha$-LIC, then 
   for each $1 \leq n \leq N$, $(\Psi^{(n)}, \Phi^{(n)})$ is a dual frame pair in $L^2(G)$ if, and only if, in view of (\ref{eq 6.5}), the relation (\ref{eq 6.4111}) holds. Moreover, 
  for $1\leq n_1 \neq n_2 \leq N$, under the same assumptions, $\mathcal{W}(\Psi^{(n_1)}, \mathcal{A}, \Gamma, \Lambda)$ and  $\mathcal{W}(\Phi^{(n_2)}, \mathcal{A}, \Gamma, \Lambda)$  are orthogonal frames if, and only if, the relation (\ref{eq 6.41111}) is satisfied
  by using (\ref{eq 6.6}).
  Thus, the proof follows from Corollary~\ref{prop 4.3}.
\end{proof}
 The following can be easily deduced from the above result. Note that it  generalizes similar  results of Labate et al. \cite{LWW} and Hern\'andez et al.\cite{HLWW} to the setting of super-spaces over LCA groups.
\begin{corollary} \label{coro 6.11} 
For each $1 \leq n \leq N$, let  $\mathcal{W}(\Psi^{(n)}, \mathcal{A}, \Gamma, \Lambda)$ be a wave-packet system in $L^2(G)$ which satisfies the corresponding $\alpha$-LIC, where $\mathcal{A}$ is a countable subset of {{\bf{Aut}}$($$G$$)$}. Then, the super wave-packet system generated by  $\bigoplus \limits_{n=1}^{N}\Psi^{(n)}$  forms  a Parseval frame for $L^2(G)^{(N)}$ if, and only if, both $($$\ref{eq 6.4111}$$)$ and $($$\ref{eq 6.41111}$$)$ hold for ${\Psi}^{(n)}={\Phi}^{(n)}; 1\leq n \leq N.$
\end{corollary}	
In the following, by applying Corollary~\ref{coro 6.11} to the case $G=\mathbb{R}^d$, we reach at the results obtained in \cite{LWW, HLW}. Hence, the wave-packet systems within $L^2(\mathbb{R}^d)$ are easily covered within our framework. 
\begin{example}\label{eg 6.3}
	Let $G=\mathbb{R}^d$ (equipped with Lebesgue measure), $\Gamma=\mathbb{Z}^d$ and $\Lambda=\mathbb{R}^d$. Then, $\widehat{G}=\mathbb{R}^d$, with Euclidean metric, we have ${\Gamma}^{\perp}=\mathbb{Z}^d$ and  ${\Lambda}^{\perp}=\{0\}$. Let $A \in$ GL($d,\mathbb{R}$) be a matrix whose eigenvalues are strictly larger than one in modulus, set  $\mathcal{A}=\{x \mapsto A^kx: k \in \mathbb{Z}\}$. Under these assumptions, from (\ref{eq 6.2}), for each $1 \leq n \leq N$, the wave-packet system generated by $\Psi^{(n)}=\{\psi^{(n)}_l\}_{l=1}^{L}\subset L^2(\mathbb{R}^d)$ can be written as 
	\begin{align*}
	\mathcal{W}(\Psi^{(n)}, \mathcal{A}, {\mathbb{Z}}^d, {\mathbb{R}}^d)&:=\big\{D_{A^k}T_{\gamma}M_{\chi}{\psi}^{(n)}_l(\cdot{}): l=1,\ldots,L,\, k \in \mathbb{Z},\gamma \in \mathbb{Z}^d, \chi \in \mathbb{R}^d \big\}\\
	&=\big\{|\det A|^{-k/2}\chi(A^{k}\cdot{}-\gamma){\psi}^{(n)}_l(A^{k}\cdot{}-\gamma): l=1,\ldots,L,\, k \in \mathbb{Z},\gamma \in \mathbb{Z}^d, \chi \in \mathbb{R}^d \big\}.
	\end{align*}
	For each $1 \leq n \leq N$, by letting $\mathcal{W}(\Psi^{(n)}, \mathcal{A}, {\mathbb{Z}}^d, {\mathbb{R}}^d)$ as a wave-packet system in  $L^2(\mathbb{R}^d)$ which satisfies the Bessel condition, we conclude from Corollary~\ref{coro 6.11} that the super wave-packet system generated by  $\bigoplus \limits_{n=1}^{N}\Psi^{(n)}$  form  a Parseval frame for $L^2(\mathbb{R}^d)^{(N)}$ if, and only if, for a.e. $\xi \in \mathbb{R}^d$ and for each $\widetilde{\alpha} \in \bigcup \limits_{k \in \mathbb{Z}} {B}^{k}\mathbb{Z}^d$ along with $B=A^{\ast}$, both of the following hold:
	\begin{itemize}	
		\item[(i)] 
		for each $1 \leq n \leq N$, 
	$
		\displaystyle\sum\limits_{k\in \mathbb{Z}:\, \widetilde{\alpha} \in {B}^{k}\mathbb{Z}^d}\sum \limits_{l=1}^{L}\int \limits_{\mathbb{R}^d} \overline{\widehat {\psi}^{(n)}_l({B}^{-k}\xi-\chi)} \widehat {\psi}^{(n)}_l({B}^{-k}(\xi+ \widetilde{\alpha})-\chi)d(\chi)=\delta_{\widetilde{\alpha},0},$ and
		\item[(ii)] for each $1 \leq n_1\neq n_2 \leq N$,
	$
		\displaystyle\sum\limits_{k\in \mathbb{Z}:\, \widetilde{\alpha} \in {B}^{k}\mathbb{Z}^d}\sum \limits_{l=1}^{L}\int \limits_{\mathbb{R}^d} \overline{\widehat {\psi}^{(n_1)}_l({B}^{-k}\xi-\chi)} \widehat {\psi}^{(n_2)}_l({B}^{-k}(\xi+ \widetilde{\alpha})-\chi)d(\chi)=0.
	$
	\end{itemize}
\end{example}
  \subsection{Special cases of Wave-Packet Systems}\label{sec 6.2}
  \noindent

   \subsubsection{Gabor Systems}\label{sec 6.21}
   
   \noindent 
   In (\ref{eq 6.2}), by assuming $\mathcal{A}=\{I_G\}$, 
   where $I_G$ denotes the identity group homomorphism on $G$, we consider the following system as a special case of wave-packet system defined in (\ref{eq 6.2}) which we call as the {\textit{Gabor system}} generated by $\Psi$: 
   \begin{equation*}\label{eq 6.33}
   \mathcal{G}(\Psi,\Gamma, \Lambda):=\{T_{\gamma}M_{\chi}\psi_j: \gamma \in \Gamma, \chi \in \Lambda, j \in J\},
   \end{equation*}
 At this juncture, it is relevant to note that the system 
 		$\mathcal{G}(\Psi,\Gamma, \Lambda)$ is a frame for $L^2(G)$ if and only if $\{M_{\chi}T_{\gamma}\psi_j: \gamma \in \Gamma, \chi \in \Lambda, j \in J\}$ is a frame for $L^2(G)$ (see  \cite[Lemma 2.4]{JL}), where the later system is termed as a {\textit{co-compact Gabor system}} in \cite{JL1}. Further,
   observe that  $\mathcal{G}(\Psi,\Gamma, \Lambda)$ is a TI system of the form $\bigcup_{j \in J}\{T_{\gamma}g_{j,p}\}_{\gamma \in \Gamma_{j},\, p \in P_j}$ 
   with
   $\Gamma_j=\Gamma$ for $j \in J\subset \mathbb{Z}$ and 
   $g_{j,p}=g_{j,\chi}=M_{\chi}\psi_j$, where $(j,p)=(j,\chi) \in J \times \Lambda$. In this case, for each $j \in J$,  $P_j=\{\chi: \chi \in \Lambda\}$ is equipped with the measure $\mu_{P_j}:={(\Delta(\alpha))}^{-1}\mu_{\Lambda}$ that satisfies the standing hypothesis. Since for TI systems dual $\alpha$-LIC is automatically satisfied, thus, Theorem~\ref{thm 6.11} and Corollary~{\ref{coro 6.11}} lead to the  following result on Gabor systems which  generalizes  \cite[Theorem 4.1]{JL}:
   \begin{proposition}
   	\label{prop 6.1}
   	For each $1 \leq n \leq N$, let the Gabor systems $\mathcal{G}(\Psi^{(n)}, \Gamma, \Lambda)$ and  $\mathcal{G}(\Phi^{(n)}, \Gamma, \Lambda)$ be Bessel families in $L^2(G)$. Then, the Gabor systems generated by  $\bigoplus \limits_{n=1}^{N}\Psi^{(n)}$ and  $\bigoplus \limits_{n=1}^{N}\Phi^{(n)}$ $($we call as super Gabor systems$)$  form dual frames for $L^2(G)^{(N)}$ if, and only if, for a.e. $\xi \in \widehat{G}$, both of the following hold:
   	\begin{itemize}	
   		\item[(i)] 
   		for each $1 \leq n \leq N$ and  $\widetilde{\alpha} \in  \Gamma^{\perp}$, we have 
   		\begin{equation}\label{eq 6.3} 
   		\sum \limits_{j \in J}\int \limits_{\Lambda} \overline{\widehat {\varphi}^{(n)}_j(\xi-\chi)} \widehat {\psi}^{(n)}_j((\xi+ \widetilde{\alpha})-\chi)d{\mu}_{\Lambda}(\chi)=\delta_{\widetilde{\alpha},0}, 
   		\end{equation}
   		\item[(ii)] for each $1 \leq n_1\neq n_2 \leq N$ and  $\widetilde{\alpha} \in {\Gamma}^{\perp}$, we have 
   		\begin{equation}\label{eq 6.4} 
   		\sum \limits_{j \in J}\int \limits_{\Lambda} \overline{\widehat {\varphi}^{(n_1)}_j(\xi-\chi)} \widehat {\psi}^{(n_2)}_j((\xi+ \widetilde{\alpha})-\chi)d{\mu}_{\Lambda}(\chi)=0.
   		\end{equation}
   	\end{itemize}
   \end{proposition}
   Using Proposition~\ref{prop 6.1}, we make the following observation. One can find similar results on Gabor systems  in different settings, for example, in \cite{LH, KL, LL,J,JL1,HLW,FHS}, and various references within.
   \begin{corollary} \label{coro 6.8} 
   	For each $1 \leq n \leq N$, let  $\mathcal{G}(\Psi^{(n)}, \Gamma, \Lambda)$ be a Gabor system in $L^2(G)$. Then, the super Gabor system generated by  $\bigoplus \limits_{n=1}^{N}\Psi^{(n)}$ forms  a Parseval  frame for $L^2(G)^{(N)}$ if, and only if, both $($$\ref{eq 6.3}$$)$ and $($$\ref{eq 6.4}$$)$ hold for ${\Psi}^{(n)}={\Phi}^{(n)}; 1\leq n \leq N.$
   \end{corollary}
   From Proposition~\ref{prop 6.1}, we can derive various results on Gabor systems by using different situations on $\Gamma$, $\Lambda$ and $G$, etc. In the following, by letting $\Gamma$ as a uniform lattice, we deduce a characterization of all the functions $\Psi$ such that the Gabor system generated by $\Psi$ forms a Parseval frame for $L^2(G)$. It turns out that this result generalizes similar works in $L^2(\mathbb{R}^d)$ (e.g., see \cite{HLW}) and $l^2(\mathbb{Z}^d)$ (e.g., see \cite{LH}).
   \begin{example}
   	\label{eg 6.61}	
   	In Corollary~\ref{coro 6.8}, let $\Gamma  \subset G$ be a uniform lattice and let $\Lambda$
   	be a discrete subset of $\widehat{G}$. Further, we assume $N=1$ and let $\Psi^{(n)}= \Psi$ for each $1\leq n \leq N$. Then, it is clear that we can write the system $\mathcal{G}(\Psi, \Gamma, \Lambda)$ in the form of a GSI system 
   	$\bigcup _{j \in J}\{T_{\gamma}g_{j,p}\}_{\gamma \in \Gamma_j,\, p \in P_j}$ 
   	with $\Gamma_j=\Gamma$ for $j \in J \subset \mathbb{Z}$ and
   	$g_{j,p}=g_{j,\chi}=M_{\chi}\psi_j$, where $(j,p)=(j,\chi) \in J \times \Lambda$. In this case, for each $j \in J$, the measure space $P_j=\{\chi: \chi \in \Lambda\}$ is equipped with the measure $\mu_{P_j}:={(\Delta(\alpha))}^{-1}\mu_{\Lambda}$ that satisfies the standing hypothesis. Now, from Corollary~\ref{coro 6.8}, we can deduce a characterization of all functions $\Psi$ such that $\mathcal{G}(\Psi, \Gamma, \Lambda)$ is a Parseval frame for $L^2(G)$. More precisely, $\mathcal{G}(\Psi, \Gamma, \Lambda)$ is a Parseval frame for 
   	$L^2(G)$ if, and only if, for each $\widetilde{\alpha} \in {\Gamma}^{\perp}$ and for a.e. $\xi \in \widehat{G}$, we have
   	\begin{equation*}
   	\sum \limits_{j \in J}\sum \limits_{\chi \in \Lambda} \overline{\widehat {\psi}_j(\xi-\chi)} \widehat {\psi}_j((\xi+ \widetilde{\alpha})-\chi)=\delta_{\widetilde{\alpha},0}. 
   	\end{equation*}
   \end{example}    
  \subsubsection{Wavelet Systems}\label{sec 6.22}
  
  \noindent 
  By letting $\Lambda=\{\chi_0\} \subset \widehat{G}$ in (\ref{eq 6.2}), where $\chi_0$ being the neutral element of $\widehat{G}$, we define the 
  collection
  $\mathcal{U}(\Psi,\mathcal{A},\Gamma )$ as the {\textit{wavelet system}} generated by $\Psi$:
  \begin{equation}\label{eq 6.22}
  \mathcal{U}(\Psi,\mathcal{A},\Gamma ):=\{D_{\alpha}T_{\gamma}\psi_j: \alpha \in \mathcal{A},\gamma \in \Gamma, j \in J\},
  \end{equation}
  as a special case of wave-packet system defined in (\ref{eq 6.2}). For a countable subset $\mathcal{A}$ in {{\bf{Epick}}($G$)}, the system (\ref{eq 6.22}) is a GTI system of the form 
  $\bigcup_{\alpha \in \mathcal{A}}\{T_{\gamma}g_{\alpha,p}\}_{\gamma \in \Gamma_{\alpha},\, p \in P_{\alpha}}$ for
  $\Gamma_{\alpha}={\alpha}^{-1}\Gamma$ with $\alpha \in \mathcal{A}$,
  $g_{\alpha, p}=g_{\alpha, j}=D_{\alpha}\psi_j$ for  $(\alpha,p)=(\alpha,j)$ in   $\mathcal{A}\times J$. In this case, for each $\alpha \in \mathcal{A}$, the measure space $P_{\alpha}:=\{j: j \in J\}$ is equipped with a counting measure $\mu_{P_{\alpha}}:= {(\Delta(\alpha))}^{-1}(\mu_{J})$ which is clearly $\sigma$-finite. Thus, Theorem~\ref{thm 6.11} and Corollary~{\ref{coro 6.11}}  for the case of wave-packet systems now reduce to the following results on wavelet systems. We mention that this result generalizes the duality results for wavelet systems investigated by various authors, including  \cite{B}, to the set-up of super-spaces over LCA groups: 
   \begin{proposition}
   	\label{coro 6.12}
   	For each $1 \leq n \leq N$, let the wavelet systems $\mathcal{U}(\Psi^{(n)}, \mathcal{A}, \Gamma)$ and  $\mathcal{U}(\Phi^{(n)}, \mathcal{A}, \Gamma)$ be Bessel families in $L^2(G)$ which satisfy the corresponding dual $\alpha$-LIC, where $\mathcal{A}$ is a countable subset of {{\bf{Aut}}$($$G$$)$}. Then, the wavelet systems generated by  $\bigoplus \limits_{n=1}^{N}\Psi^{(n)}$ and  $\bigoplus \limits_{n=1}^{N}\Phi^{(n)}$   $($we call as {\textit{super-wavelet systems}}$) $  form  dual  frames for $L^2(G)^{(N)}$ if, and only if, for a.e. $\xi \in \widehat{G}$, both of the following hold:
   	\begin{itemize}	
   		\item[(i)] 
   		for each $1 \leq n \leq N$ and  $\widetilde{\alpha} \in \bigcup \limits_{\alpha \in \mathcal{A}} \Gamma^{\perp}_{\alpha}$, we have
   		\begin{equation}\label{eq 6.2222} 
   			\sum\limits_{\alpha \in \mathcal{A}:\, \widetilde{\alpha} \in \Gamma^{\perp}_{\alpha}}\sum \limits_{j \in J} \overline{\widehat {\varphi}^{(n)}_j({\beta}^{-1}\xi)} \widehat {\psi}^{(n)}_j({\beta}^{-1}(\xi+ \widetilde{\alpha}))=\delta_{\widetilde{\alpha},0}, 
   		\end{equation}
   		\item[(ii)] for each $1 \leq n_1\neq n_2 \leq N$ and  $\widetilde{\alpha} \in \bigcup \limits_{\alpha \in \mathcal{A}} \Gamma^{\perp}_{\alpha}$, we have 
   		\begin{equation}\label{eq 6.42222} 
   			\sum\limits_{\alpha \in \mathcal{A}:\, \widetilde{\alpha} \in \Gamma^{\perp}_{\alpha}}\sum \limits_{j \in J} \overline{\widehat {\varphi}^{(n_1)}_j({\beta}^{-1}\xi)} \widehat {\psi}^{(n_2)}_j({\beta}^{-1}(\xi+ \widetilde{\alpha}))=0,
   		\end{equation}
   	\end{itemize}
   	 where for $\beta={\alpha}^
   	 {\ast}$, $\Gamma^{\perp}_{\alpha}$ is given by $\beta\Gamma^{\perp}$.
   \end{proposition}
   The following result generalizes \cite[Theorem 1.7]{W}, and can be easily derived from Proposition~\ref{coro 6.12}:
   \begin{corollary} \label{coro 6.13} 
   	For each $1 \leq n \leq N$, let  $\mathcal{U}(\Psi^{(n)}, \mathcal{A}, \Gamma)$ be a wavelet system for $L^2(G)$  which satisfies the corresponding $\alpha$-LIC, where $\mathcal{A}$ is a countable subset of {{\bf{Aut}}$($$G$$)$}. Then, the super-wavelet system generated by  $\bigoplus \limits_{n=1}^{N}\Psi^{(n)}$  forms a Parseval frame for $L^2(G)^{(N)}$ if, and only if, both $($$\ref{eq 6.2222}$$)$ and $($$\ref{eq 6.42222}$$)$ hold for ${\Psi}^{(n)}={\Phi}^{(n)}; 1\leq n \leq N.$
   \end{corollary}
   	\begin{example}\label{eg 6.6}
   	By assuming $\Lambda=\{\chi_0\} \subset \widehat{G}$ in Example~\ref{eg 6.3}, where $\chi_0$ being the neutral element of $\widehat{G}$,  for each $1 \leq n\leq N$, we obtain a wavelet system generated by $\Psi^{(n)}$:
   		\begin{align*}\
   		\mathcal{U}(\Psi^{(n)},\mathcal{A},{\mathbb{Z}}^d )&=\big\{D_{A^k}T_{\gamma}{\psi}^{(n)}_l(\cdot{}): l=1,\ldots,L,\, k \in \mathbb{Z}, \gamma \in \mathbb{Z}^d \big\}\\
   		&=\big\{|\det A|^{-k/2}{\psi}^{(n)}_l(A^{k}\cdot{}-\gamma): l=1,\ldots,L,\, k \in \mathbb{Z}, \gamma \in \mathbb{Z}^d \big\},
   		\end{align*}
   	which is a special case of wave-packet system  $\mathcal{W}(\Psi^{(n)}, \mathcal{A}, {\mathbb{Z}}^d, {\mathbb{R}}^d)$.	It follows that two Bessel families  $\mathcal{U}(\Psi^{(n)},\mathcal{A},{\mathbb{Z}}^d )$ and $\mathcal{U}(\Phi^{(n)},\mathcal{A},{\mathbb{Z}}^d )$ are 
   	\begin{itemize}
   	\item[(a)] dual frames if, and only if, for each $1 \leq n \leq N$, 
   		 \begin{equation*}
  \sum\limits_{k\in \mathbb{Z}:\, \widetilde{\alpha} \in {B}^{k}\mathbb{Z}^d}\sum \limits_{l=1}^{L} \overline{\widehat {\varphi}^{(n)}_l({B}^{-k}\xi)} \widehat {\psi}^{(n)}_l({B}^{-k}(\xi+ \widetilde{\alpha}))=\delta_{\widetilde{\alpha},0},\,\,~\mbox{for a.e.}~\,\, \xi \in \mathbb{R}^d, 
   			\end{equation*}
   		\item[(b)] orthogonal frames if, and only if, for each $1 \leq n_1\neq n_2 \leq N$,  
   		\begin{equation*}
   		\sum\limits_{k\in \mathbb{Z}:\, \widetilde{\alpha} \in {B}^{k}\mathbb{Z}^d}\sum \limits_{l=1}^{L} \overline{\widehat {\varphi}^{(n_1)}_l({B}^{-k}\xi)} \widehat {\psi}^{(n_2)}_l({B}^{-k}(\xi+ \widetilde{\alpha}))=0,\,\,~\mbox{for a.e.}~\,\, \xi \in \mathbb{R}^d, 
   		\end{equation*}
   		\end{itemize}
   for all  $\widetilde{\alpha} \in \bigcup \limits_{k \in \mathbb{Z}} {B}^{k}\mathbb{Z}^d$. Clearly, the wavelet systems generated by  $\bigoplus \limits_{n=1}^{N}\Psi^{(n)}$ and  $\bigoplus \limits_{n=1}^{N}\Phi^{(n)}$     form  dual  frames for $L^2(\mathbb{R}^d)^{(N)}$ if, and only if, both of the equalities (a) and (b) are satisfied. Here, note that the results obtained in Example~\ref{eg 6.6}(a)
   and Example~\ref{eg 6.6}(b) coincide with the characterizations of two wavelet systems to be dual frames (e.g., see \cite{B}) and orthogonal 	frames \cite{W}, respectively.
   	\end{example}
  \begin{remark}\label{re 6.10} Similar to the case of Gabor systems considered in Example~\ref{eg 6.61}, we can study duals of wavelet systems with translations along uniform lattices as a special case of Proposition~\ref{coro 6.12}. For defining such systems, we can use an approach similar to  Dhalke \cite{D}, and Kutyniok and Labate \cite{KL}, where the dilations have been treated as expensive automorphisms on an LCA group $G$.
  \end{remark}


\end{document}